\documentclass[a4paper]{amsart}
\usepackage{amsmath,amssymb,graphics,verbatim,eucal,latexsym}

\makeatletter
\@addtoreset{equation}{section}\makeatother

\newtheorem{theorem}{Theorem}[section]
\newtheorem{proposition}[theorem]{Proposition}  
\newtheorem{lemma}[theorem]{Lemma}

\newtheorem{corollary}[theorem]{Corollary}
\newtheorem{definition}[theorem]{Definition}
\newtheorem{remark}[theorem]{Remark}

\def\phi{\varphi}
\def\>{\rangle}
\def\<{\langle}
\def\R{\mathbb{R}}

\def\N{\mathbb{N}}

\def\Ric{\mathop{\rm Ric}\nolimits}

\def\Diam{\mathop{\rm \delta}\nolimits}
\def\Vol{\mathop{\rm Vol}\nolimits}
\def\exp{\mathop{\rm exp}\nolimits}
\def\inj{\mathop{\sl i_M}\nolimits}

\def\dis{\displaystyle}

\def\dsl{\displaylines}
\def\pl{\partial}

\begin{document}
\title[Discretization of a compact manifold and spectrum]{Approximation of the spectrum of a manifold by discretization}
\author{Erwann AUBRY}
\address{Laboratoire Dieudonn\'e\\
Universit\'{e} de Nice Sophia-Antipolis\\
Parc Valrose\\
06108 Nice Cedex\\ 
FRANCE}
\email{eaubry@math.unice.fr}
\date{}
\begin{abstract} We approximate the spectral data (eigenvalues and eigenfunctions) of compact Riemannian manifold by the spectral data of a sequence of (computable) discrete Laplace operators associated to some graphs immersed in the manifold. We give an upper bound on the error that depends on upper bounds on the diameter and the sectional curvature and on a lower bound on the injectivity radius.\end{abstract}
\maketitle

\section{Introduction}

We prove that the spectral data (eigenvalues and eigenfunctions) of any closed Riemannian manifold can be approximated by the corresponding spectral data of the Laplace operator of some graphs geodesically immersed in the manifold. It is an extension of the finite elements method to the Riemannian setting. The two main points of our method are the following.
\begin{enumerate}
\item The error made on the spectral data are bounded above by universal functions of some bounds on the geometry of the manifold (i.e. bounds on the injectivity radius, the sectional curvature and the diameter) and of the graph (i.e. bounds on the thinness and mesh of the graph). This errors tend to $0$ as the mesh of the graph tends to $0$.
\item The Laplace operator of a metric graph is a universal and explicitly computable function of its edge-lengths.
\end{enumerate}
Before stating our main results, we need a few definitions and notations. 

\subsection{Definitions and notations}
We will work with a special kind of immersed graphs, that we will call {\sl geodesic triangulations} (see the definition in Section \ref{geodtriang}). They are not necessarily actual triangulations of $M$ (for instance the simplices of dimension greater than $1$ are not necessarily realized as subset of $M$) but are more easier to construct.

A geodesic triangulation $T$ of an $n$-dimensional Riemannian manifold $(M^n,g)$ is a set of points $(x_i)_{1\leqslant i\leqslant N}$ of $M$ endowed with a structure of abstract simplicial complex $K$ which satisfies the two properties \ref{ax1} and \ref{ax2} of section \ref{geodtriang}.
We denote by $S_p$ the set of closed $p$-simplices of $K$. We identify the edges of $K$ with some minimizing, geodesic segment between their vertexes.
For any closed simplex $\sigma$ of $K$, we set $St(\sigma)$ (resp. $St_p(\sigma)$)  the set of the closed simplices (resp. of dimension $p$) of $K$ that contains $\sigma$. 
The vertices of any $\sigma\in S_p$ are naturally ordered by their indices ($\sigma=\{x_{i_\sigma(0)},\ldots,x_{i_\sigma(p)}\}$ with $i_\sigma(0)<\ldots<i_\sigma(p)$).
We set $X_\sigma=x_{i_{\sigma(0)}}$ and for any distinct $0\leq k\leq n$, we set $v^\sigma_{k}$ a vector of $T_{X_\sigma}M$ such that $x_{i_\sigma(k)} =\exp_{X_\sigma}(v^\sigma_{k})$.
We set also $A^\sigma$ the associated Gramm matrix $\Bigl(g(v_{k}^\sigma,v_{l}^\sigma)\Bigr)_{\tiny  \begin{matrix}
    1\leq l\leq n\\ 1\leq k\leq n
  \end{matrix}}$.
Given a geodesic triangulation of $(M^n,g)$, we note $m_T$ its {\sl mesh} (the maximal length of its edges) and $\Theta_T$ its {\sl thinness,} i.e the quantity
$$\Theta_T=\max\Bigl(\dis\max_{\tiny\begin{matrix}
    \sigma\in S_n\\0\leq k\leq n
  \end{matrix}}m_T(\det A^\sigma)^{-\frac{1}{2n}},\max_{(e_1,e_2)\in S_1}\frac{length(e_1)}{length(e_2)}\Bigr).$$
Eventually, on the set $\R^N$ of functions $y:T\to\R$, (where  $N$ is the number of vertices of $T$ and we identify $T$ with $S_0$), we define two quadratic forms by the formulae
\begin{equation}
  \label{L2disc}
  |y|_T^2=\frac{2}{(n+2)!}\sum_{1\leq i\leq j\leq N} y_iy_{j}\sum_{\sigma\in St_n([x_i,x_{j}])}\sqrt{\det A^\sigma},
\end{equation}
\begin{equation}
  \label{Ddisc}
  q_T(y)=\frac{1}{n!}\sum_{\sigma\in S_n}\sqrt{\det A^\sigma}\sum_{k,l=1}^n(A^\sigma)^{kl}(y_{i_\sigma(k)}{-}y_{i_\sigma(0)})(y_{i_\sigma(l)}{-}y_{i_\sigma(0)}).
\end{equation}
Note that if $K$ is a sub-complex of $\R^n$ then $|y|_T$ and $q_T(y)$ give respectively the $L^2$-norm and Dirichlet energy of the affine-by-parts expansion of $y$.

\subsection{Main results}
For any closed, Riemannian $n$-manifold $(M,g)$, we denote by $\delta_M$ its diameter, by $R$ an upper bound of all its sectional curvatures and by $\inj$ its injectivity radius. We denote also by $0=\lambda_0(T)\leq\cdots\leq\lambda_{N-1}(T)$ the eigenvalues of $q_T$ with respect to $|\cdot|^2_T$ and $0=\lambda_0(M)<\lambda_1(M)\leq\cdots\leq\lambda_i(M)\leq\cdots$ the eigenvalues of $(M^n,g)$.

\begin{theorem}\label{approvalp}
Let $n\geq2$ be an integer, and $\epsilon\in]0,1[$ be a real number. There exists a constant $C(n)$ such that if 
\begin{enumerate}
\item[i)] $(M^n,g)$ is a closed, Riemannian $n$-manifold which satisfies $\Diam_M^2|R|\leq\Lambda^2$, 
\item[ii)] $T$ is a geodesic triangulation  of $M$ which satisfies  $\frac{m_T}{\Diam_M}\leq C(n)\bigl(\frac{\inj}{\Diam_M\Theta_Te^{e^\Lambda}p}\bigr)^{3n^3}\epsilon$,
\end{enumerate}
then we have
$$\displaylines{(1-\epsilon)\lambda_p(T)\leq\lambda_p(M)\leq(1+\epsilon)\lambda_p(T).}$$
\end{theorem}

\begin{remark}
  The constant $C(n)$ is computable. The second condition says that any finite number of eigenvalues can be approximated provided the mesh of the graph is small enough and the thinness is controlled.
\end{remark}

\begin{remark} The matrices $A^\sigma$ depend on the angle between some edges of $T$ issued from a same vertex, but the same result is valid if we replace the coefficient $g(v_{k}^\sigma,v_{l}^{\sigma})$ by $$\frac{1}{2}\Bigl[d^2(X_\sigma,x_{i_\sigma(k)})+d^2(X_\sigma,x_{i_\sigma(l)})-d^2(x_{i_\sigma(k)},x_{i_{\sigma(l)}})\Bigr]$$
in the definition of the matrix $A^\sigma$.
This gives approximation of the eigenvalues of $M$ by the eigenvalues of a discrete Laplace operator whose coefficients are universal functions of the lengths of a geodesically immersed graph of $M$.

Note that in \cite{BIK}, the authors get the same result for another geometric quadratic form $q_T$, whose coefficients depend on the volume of the Vorono\"i cells associated to a lattice $(x_i)_{i\in I}$ which need not to be the vertices of a geodesic triangulation. 
\end{remark}

We denote by $(f_i^T)$ the eigenvectors of $q_T$ with respect to $|\cdot|^2_T$ and let $(f_i)_{i\in \N}$ be a $L^2$ orthonormal family of eigenfunctions of $(M^n,g)$ such that $\Delta f_i=\lambda_i f_i$ for all $i\in\N$. For some integers $p<q$, we set $E_{p,q}$ (resp. $F_{p,q}$) the sum of the eigenspaces of $\Delta(M)$ (resp. $q_T$) associated to the eigenvalues $(\lambda_i(M))_{p+1\leqslant i\leqslant q}$ (resp. $(\lambda_i(T))_{p+1\leqslant i\leqslant q}$ and $P_{p,q}$ (resp. $Q_{p,q}$) the normal projection on $E_{p,q}$ (resp. $F_{p,q}$).

\begin{theorem}\label{approvecp}
Under the assumptions of Theorem \ref{approvalp}, if there exist some integers $p<q$ and $\eta>0$ such that
$\lambda_p+\eta\leqslant\lambda_{p+1}$ and  $\lambda_q+\eta\leqslant\lambda_{q+1}$,  then for any $f\in E_p$, we have
$\|R(f)-P_{p,q}\circ R(f)\|_T^2\leqslant \frac{C(q,n,\Lambda,\frac{\Diam_M}{i_0})}{\sqrt{\eta}}(\frac{m_T}{\Diam_M})^\frac{1}{6n^2}\|R(f)\|_T^2$,
and for any $(y_i)\in F_p$, we have
$\|W(y_i)-Q_{p,q}\circ W(y_i)\|_T^2\leqslant C(q,n,\Lambda,\frac{\Diam_M}{i_0},\eta)(\frac{m_T}{\Diam_M})^\frac{1}{6n^2}\|W(y_i)\|_T^2$.
\end{theorem}

To get an approximation of the spectral data of $(M^n,g)$ by those of $q_T$ in Theorems \ref{approvalp} and \ref{approvecp}, we need some geodesic triangulations with arbitrary small mesh but bounded thinness. The existence of such fat triangulation is often admitted or conjectured but we do not know complete published proof of this fact. For sake of completeness, we give a constructuve proof of the following result (based on the previous work of J.Cheeger, S.M\"uller and R.Schr\"ader).

\begin{theorem}\label{triageod}
  Let $n\geq2$ be an integer and $D,i_0$ and $\Lambda$ be some positive real numbers. There exist some constants $\beta(n)$ and $C(i_0/D,\Lambda,n)$ such that for any Riemannian manifold $(M^n,g)$ with diameter $\Diam_M\leqslant D$, sectional curvature $\Diam_M^2|R|\leqslant\Lambda^2$ and injectivity radius $\inj\geqslant i_0$ and for any $\epsilon\in]0,C(i_0/D,\Lambda,n)[$, there exists a geodesic triangulation $T$ of $M$ with mesh $m_T\leq\epsilon$ and thinness $\Theta_T\leq 1/C(i_0/D,\Lambda,n)$.
\end{theorem}

\begin{remark}
Once again the constants of Theorem \ref{triageod} are explicitly computable.
Combining Theorems \ref{triageod} and \ref{approvalp}, for any compact manifold $(M^n,g)$, any $N\in\N$ and any $\varepsilon>0$ we get a method to construct a geodesic triangulation $T$ of $M$ such that we have $(1-\epsilon)\lambda_p(T)\leq\lambda_p(M)\leq(1+\epsilon)\lambda_p(T)$ for any $p\leqslant N$.
\end{remark}

\subsection{Main steps of the proof}
Let $(E,\langle\cdot\,,\cdot\rangle)$ be a Euclidean space endowed with a bilinear symmetric form $q$, and $\lambda_0\leq\cdots\leq\lambda_{dim E-1}$ be the eigenvalues of $q$ with respect to $\<\cdot\,,\cdot\>$. Using the min-max principle we readily infer the following spectral comparison principle.

\begin{proposition}[small eigenvalue principle]\label{min-max}
  Let $\bigl(E_1,\langle\cdot,\cdot\rangle_1\bigr)$ and $\bigl(E_2,\<\cdot,\cdot\>_2\bigr)$ be two Euclidean spaces endowed respectively with quadratic forms $q_1$, $q_2$.
If there exists a linear map $\Phi:E_1\to E_2$ and two positive real numbers $\alpha$, $\beta$ such that
$$\displaylines{\hfill\<\Phi(x),\Phi(x)\>_2\,\geq\,\alpha\<x,x\>_1\hfill\mbox{and}\hfill
q_2\bigl(\Phi(x),\Phi(x)\bigr)\leq\beta\,q_1(x,x),\hfill}$$
then we have $\lambda_k(q_2)\leq\frac{\beta}{\alpha}\lambda_k(q_1)$ for any $k$.
\end{proposition}

Proposition \ref{min-max} is usually used to compare spectra under small perturbations on the metric or on the manifold. It is the key tool of our eigenvalues approximation method.

Given a manifold $M$ and a geodesic triangulation $T$ of $M$, we denote by $(x_i)_{1\leq i\leq N}$ the vertices of a $T$, by $E_p$ the subspace of $H^{1,2}(M)$ spanned by the $p+1$ first eigenfunctions $(f_i)_{0\leq i\leq p}$ of $M$, by  $R:E_p\to\R^N$ the natural restriction map $R(f)=(f(x_i))_{1\leq i\leq N}$, by $\<\cdot,\cdot\>$ the scalar product on $E_p$ induced by the $L^2$-norm on $M$, and we set $q(f)=\int_M|df|^2$. The spectrum of $q$ with respect to $\<\cdot,\cdot\>$ is given by the $p+1$ first eigenvalues (counted with multiplicities) of $M$. We then proceed in two steps.

\begin{enumerate}
\item 
A Moser's iteration scheme gives bounds of the quotients $\frac{\|\nabla df\|_\infty}{\|f\|_2}$ on $E_p\setminus\{0\}$ by a universal function of $\lambda_p$, $\Diam$ and $\Lambda$ (see proposition \ref{estimapprox}). This Hessian bounds imply the following estimates (see Propositions \ref{DiscL2} and \ref{Diri1})
\begin{align}
&\Bigl|\<R(f),R(f)\>_T-\int_Mf^2\Bigr|\leq C m_T\int_Mf^2,\\ &q_T\bigl(R(f)\bigr)\leq (1+Cm_T)\int_M|df|^2,
\end{align} 
for any $f\in E_p$, where $|\cdot|_T$ and $q_T$ are the discrete quadratic forms on $\R^N$ given by the formulae \eqref{L2disc} and \eqref{Ddisc}, and where $C$ is a constant which depends on bounds on $\lambda_p$ and on the geometries of $(M^n,g)$ and $T$. Proposition \ref{min-max} gives then some lower bounds on the spectrum of $(M^n,g)$ of the form (see Theorem \ref{MinLi}) $\lambda_k(M)\geq\bigl(1-Cm_T\bigr)\lambda_k(T)$ for any $k\leq p.$

\item In Section \ref{triang}, we construct an expansion (Withney) map $W:\R^N\to C^\infty(M)$ such that $R\circ W={\rm Id}_{\R^N}$ and which satifies the following estimates
\begin{align}
&\Bigl|\|(y_i)\|^2_T-\int_MW(y_i)^2\Bigr|\leq Cm_T\bigl(\|(y_i)\|^2_T+q_T^{~}(y_i)\bigr),\\
&\int_M|dW(y_i)|^2\leq\bigl(1+Cm_T\bigr)q_T^{~}(y_i),
\end{align}
for any $(y_i)\in\R^N$. From Proposition \ref{min-max} again we infer that for any $k\leq p$ we have $\bigl(1-Cm_T\bigr)\lambda_k(M)\leq\lambda_k(T)$ (see Theorem \ref{MaxLi}). The construction of the withney map is the main technical difficulty of the proof. It is done by local mean of the affine expansions obtained by identifying the simplex of the geodesic triangulation with Euclidean simplicex through the Riemannian exponential maps at the vertices of the simplex.
\end{enumerate}

Note that J.~Dodziuk \cite{Do-Pa} developed  another generalization of the finite element method to compact Riemannian manifolds in which, to any smooth triangulation of $(M,g)$ is associated the subspace of $H^{1,2}(M^n,g)$ of the continuous functions on $M$ which are affine on each simplex (this subspace has finite dimension). The authors consider on it the quadratic forms induced by $\|\cdot\|_2$ and the ambient Dirichlet form $q(f)=\int_M|df|^2$ . They prove that the spectrum of the discrete Dirichlet form with respect to the the discrete $L^2$ norm converges to the spectrum of $(M^n,g)$ when the mesh of the triangulations tend to $0$ with controlled thinness. However, they do not prove that the error is bounded by geometrical bounds on the manifolds, and moreover, the discrete quadratic forms cannot be explicitly computed as function of the geometric data (edge's lengths, edge's angle) of the triangulations. 

The proof of Theorem \ref{approvecp} is done in section \ref{compeigenf} using the above estimates and the technique developed by Y. Colin de Verdi\`ere in \cite{CV}.
\smallskip

{\sl Aknowledgement} We thank S.Gallot for fruitful discussions and C.Vernicos for bringing our attention to the paper \cite{BIK}.

\section{Geodesic triangulations}\label{geodtriang}
\subsection{Definition}
A geodesic triangulation $T$ of an $n$-dimensional Riemannian manifold $(M^n,g)$ is a set of points $(x_i)_{1\leqslant i\leqslant N}$ of $M$ endowed with a structure of abstract $n$-dimensional simplicial complex $K$ whose simplices are all contained in a $n$-dimensional simplex of $K$ and which satisfies two more properties for which we need to complete the notations of the introduction.

For $\sigma\in S_p$ and any distinct $0\leq k,l\leq p$, $v^\sigma_{kl}$ is a vector of $T_{x_{i_\sigma(k)}}M$ such that $x_{i_\sigma(l)}=\exp_{x_{i_\sigma(k)}}(v^\sigma_{kl})$ and $|v^\sigma_{kl}|=L^\sigma_{kl}=d(x_{i_\sigma(k)},x_{i_\sigma(l)})$.  We set $A^\sigma_k$ the matrix $\Bigl(g(v_{kl}^\sigma,v_{km}^\sigma)\Bigr)_{\tiny  \begin{matrix}
    l\neq k\\ m\neq k
  \end{matrix}}$. For any vertex $x$ of $T$ and any $\sigma\in St(x)$, $C_\sigma$ is the cone of $T_kM$ spanned by the vectors $(v^\sigma_{(i^\sigma)^{-1}(x)l})_{0\leq l\leq p}$. Note that for any $\sigma\in S_n$, we have $X_\sigma=x_{i_\sigma(0)}$, $L^\sigma_k=L_{0k}^\sigma$ and $A^\sigma=A^\sigma_{0}$ (according to the definitions given in the introduction section). For any simplex $\sigma$ of $K$, we denote by $N_p(\sigma)$ the set of all the simplices of dimension $p$ that intersect $\sigma$.

We set $F=\{(\theta_i)\in\R^{n+1}/\,\sum_i\theta_i=1\}$. For any $\lambda\in\R$, $\Delta_\lambda^n$ is the closed $n$-simplex $F\cap[1-\lambda,+\infty[^{n+1}$ (we will denote $\Delta^n=\Delta_1^n$). Given $\sigma\in S_n$ and $0\leq k\leq n$, we get some local {\sl barycentric coordinates} on $M$ by the formula
$$B^\sigma_k:(\theta_i)\in F\mapsto\exp_{x_{i_\sigma(k)}}\Bigl(\sum_{l\neq k}\theta_l v_{kl}^\sigma\Bigr).$$
We set $\lambda\cdot T_\sigma=B^\sigma_{0}(\Delta_\lambda ^n)$.
Eventually, a geodesic triangulation $T$ has to satisfy the following two properties.

\begin{equation}\label{ax1}
\mbox{For any vertex $x$ of $T$, $(C_\sigma)_{\sigma\in St(x)}$ induces a triangulation of the unit sphere of $T_{x}M$.}
\end{equation}
\begin{equation}\label{ax2}
\mbox{For any disjoint $\sigma, \sigma'\in S_n$ and  any $0\leq k,k'\leq n$, we have $B_k^\sigma(\Delta^n)\cap B_{k'}^{\sigma'}(\Delta^n)=\emptyset$.}
\end{equation}

Eventually, a geodesic triangulation with boundary $T$ of an $n$-dimensional Riemannian manifold $(M^n,g)$ is a set of points $(x_i)_{1\leqslant i\leqslant N}$ of $M$ endowed with a structure of abstract $n$-dimensional simplicial complex $K$ whose simplices are all contained in a $n$-dimensional simplex of $K$ and which satisfies condition \eqref{ax2} but condition \eqref{ax1} only for vertexes not on the boundary of $K$, where we call boundary of $K$ the complex of the simplices of $K$ that are contained in a $n-1$ simplex of $K$ itself contained in only one $n$-dimensional simplex of $K$.

\subsection{Metric estimates}

We now study some metric properties of the geodesic triangulations in bounded geometry.
We first recall some estimates on the Riemannian exponential map whose proofs can be found in \cite{BuK}.

\begin{theorem}\label{BuserKarcher}
  Let $(M^n,g)$ be a compact, Riemannian manifold with $\delta^2_M|\sigma|\leq\Lambda^2$.

Let $v\in T_xM$ be fixed, $y=\exp_x(v)$ and for any $w\in T_xM$, let $w(t)$ be the parallel translation of $w$ along $t\mapsto \exp_x(tv)$. If we define two maps from $T_xM$ to $M$ by $F(w)=\exp_x(v+w)$ and $G(w)=\exp_y\bigl(w(1)\bigr)$, then they satisfy the following estimates
\begin{align*}
&d\bigl(F(w),G(w)\bigr)\leq\frac{1}{3}d(x,y)|w|_{g_x}\frac{\Lambda}{\delta_M}\sinh\bigl(\frac{\Lambda}{\delta_M}(d(x,y)+|w|_{g_x})\bigr),\\
&|d_v\exp_x(w)-w(1)|\leq|w|\bigl(\frac{\sinh(\Lambda\frac{|v|}{\delta_M})}{\Lambda\frac{|v|}{\delta_M}}-1\bigr).
\end{align*}
\end{theorem}

\begin{theorem}\label{Lipexp}
   Let $(M^n,g)$ be a manifold with $\delta_M^2|\sigma|\leq\Lambda^2$ and $\epsilon<\inf\bigl(\inj,\frac{\Diam_M}{2\Lambda}\bigr)$ be a positive real. Then for any $x\in M$, the map $\exp_x$ is a diffeomorphism from $B(0_x,\epsilon)\subset T_xM$ to $B(x,\epsilon)$ and for any $u,v\in B(0_x,\epsilon)$ we have that
$$\bigl(1-\Lambda^2(\frac{\epsilon}{\delta_M})^2\bigr)|u-v|_{g_x}\leq d\bigl(\exp_x(u),\exp_x(v)\bigr)\leq\bigl(1+\Lambda^2(\frac{\epsilon}{\delta_M})^2\bigr)|u-v|_{g_x}.$$
\end{theorem}
\medskip

If $T$ is a geodesic triangulation of $(M^n,g)$ with mesh smaller than $\inj/10$, then for any $\sigma\in S_n$ and any $0\leq k\leq n$, the map $B^\sigma_k$ gives some coordinates on a neighbourhood of $\Delta_{10}^n$. We can compare these coordinates for fixed $\sigma$ but different values of $k$.

\begin{lemma}\label{coorbar}
  Let $(M^n,g)$ be a manifold with $\delta_M^2|\sigma|\leq\Lambda^2$ and $T$ be a geodesic triangulation with mesh $10\,m_T\leq\inf\bigl(\inj,\frac{\delta_M}{2\Lambda}\bigr)$. For any $\sigma\in S_n$, any $0\leq k_1,k_2\leq n$ and any $(\theta_i)\in\Delta^n_{10}$, we have
$$d\bigl(B^\sigma_{k_1}(\theta_i),B^\sigma_{k_2}(\theta_i)\bigr)\leq10\,\Lambda^2(\frac{m_T}{\delta_M})^2m_T\sum_i|\theta_i|.$$
\end{lemma}

\begin{proof}
  Let $(\theta_i)\in\Delta^n_{10}$. We set 
$v=v^\sigma_{k_1k_2}$ and $ w_{\theta}=\sum_l\theta_l(v^\sigma_{k_1l}-v^\sigma_{k_1k_2}).$
Then we have $B^\sigma_{k_1}(\theta_i)=\exp_{x_{i_\sigma(k_1)}}(v+w_\theta)$. If $w_\theta(t)$ is the parallel transport of $w$ along $s\mapsto\exp_{x_{i_\sigma(k_1)}}(sv^\sigma_{k_1k_2})$ then Theorem \ref{BuserKarcher} implies that
$$\displaylines{d\bigl(B^\sigma_{k_1}(\theta_i),\exp_{x_{i_\sigma(k_2)}}(w_\theta(1))\bigr)=d\bigl(\exp_{x_{i_\sigma(k_1)}}(v+w_\theta),\exp_{x_{i_\sigma(k_2)}}(w_\theta(1))\bigr)\leq(\frac{\Lambda}{\delta_M})^2m_T^3\sum_i|\theta_i|.}$$
For $\theta_i=\delta_{il}$ we get
$$\displaylines{d\bigl(B^\sigma_{k_1}(\delta_{il}),\exp_{x_{i_\sigma(k_2)}}(w_{\delta_{il}}(1))\bigr)=d\bigl(\exp_{x_{i_\sigma(k_2)}}(v^\sigma_{k_2l}),\exp_{x_{i_\sigma(k_2)}}(w_{\delta_{il}}(1))\bigr)\leq(\frac{\Lambda}{\delta_M})^2m_T^3,}$$
so, by Theorem \ref{Lipexp}, we have $|u^\sigma_{k_2l}-v^\sigma_{k_2l}|_{g_{x_{i_\sigma(k_2)}}}\leq2(\frac{\Lambda}{\delta_M})^2m_T^3,$
where $u^\sigma_{k_2l}$ is the parallel transport from $x_{i_\sigma(k_1)}$ to $x_{i_\sigma(k_2)}$ of the vector $v^\sigma_{k_1l}-v^\sigma_{k_1k_2}$. Hence we get
\begin{align*}
d\bigl(B^\sigma_{k_1}(\theta_i),B^\sigma_{k_2}(\theta_i)\bigr)&\leq d\bigl(B^\sigma_{k_1}(\theta_i),\exp_{x_{i_\sigma(k_2)}}(w_\theta(1))\bigr)+d\bigl(\exp_{x_{i_\sigma(k_2)}}(w_\theta(1)),B^\sigma_{k_2}(\theta_i)\bigr)\cr
&\leq(\frac{\Lambda}{\delta_M})^2m_T^3\sum_i|\theta_i|+d\bigl(\exp_{x_{i_\sigma(k_2)}}(\sum_l\theta_lu^\sigma_{k_2l}),\exp_{x_{i_\sigma(k_2)}}(\sum_l\theta_lv^\sigma_{k_2l})\bigr)\cr
&\leq(\frac{\Lambda}{\delta_M})^2m_T^3+3(\frac{\Lambda}{\delta_M})^2m_T^3\sum_l|\theta_l|.
\end{align*}
\end{proof}

For any $\tau\in S_p$, with $1\leq p\leq n-1$, we set
$$\dsl{T'_{\tau}=\bigcup_{\sigma\in St_n(\tau)}\cup_{k=0}^p \exp_{x_{i_\tau(k)}}\Bigl(\{\sum_{j=0}^p\theta_jv^\sigma_{ki_\sigma^{-1}(x_{i_\tau(j)})}/\,\theta_j\geq0,\,\sum\theta_j\leq1\}\Bigr),\cr
T_{\tau}=\bigcup_{\sigma\in St_n(\tau)}\cup_{k=0}^n\exp_{x_{i_\sigma(k)}}\Bigl(\{\sum_{j=0}^p\theta_jv^\sigma_{ki_\sigma^{-1}(x_{i_\tau(j)})}/\,\theta_j\geq0,\,\sum\theta_j\leq1\}\Bigr).}$$
For any subset $A\subset M$, we set $B(A,r)$ the tubular neighbourhood of $A$ and radius $r$. If $A$ is empty we set $B(A,r)=\emptyset$.
\medskip

Lemma \ref{coorbar} implies the following result.

\begin{corollary}\label{discret}
There exists a constant $C(n)$ such that if $(M^n,g)$ and $T$ satisfy $\delta_M^2|\sigma|\leq \Lambda^2$ and $10m_T\leq\inf\bigl(\inj,\frac{\delta_M}{C(n)\Theta_T^{2n}\Lambda}\bigr)$ then we have the following properties
  \begin{enumerate}
  \item $\Diam_{T_\sigma}^{~}\leq m_T\bigl(1+10(\frac{\Lambda}{\delta_M})^2m_T^2\bigr)$ for any $\sigma\in S_n$,
  \item the $\bigl((1-\eta)\cdot T_\sigma\bigr)_{\sigma\in S_n}$ are disjoints and the $\bigl((1+\eta)\cdot T_\sigma\bigr)_{\sigma\in S_n}$ cover $M$,
  \item for any $\sigma,\tau\in K$ such that $\sigma\cap\tau=\emptyset$, we have $d(T_\sigma,T_\tau)\geq\frac{m_T}{C(n)\Theta_T^{2n}}$,
  \item for any $\sigma\in S_p$ and $\tau\in K$, the tubular neighbourhoods $B\Bigl(T_\tau,\alpha^{p+1}\frac{m_T}{\Diam_M}m_T\Bigr)$ and $B\Bigl(T_{\sigma}\setminus B\bigl(T_{\sigma\cap \tau},\alpha^{p}\frac{m_T}{\Diam_M}m_T\bigr),\alpha^{p+1}\frac{m_T}{\Diam_M}m_T\Bigr)$ are disjoint,
\end{enumerate}
where $\eta=C(n)\,\Theta_T^{2n}(\frac{\Lambda}{\delta_M})^2m_T^2$ and $\alpha=\frac{1}{C(n)\Theta_T^{2n}}$.
\end{corollary}

\begin{proof}
We set $m=m_T$ and $\Theta=\Theta_T$.
Since the points $x_{i_\sigma(k)}$ are in $B\bigl(X_\sigma,m\bigr)$, the n-simplex $\Delta_\sigma=(0_{X_\sigma},v^\sigma_{01},\ldots,v^\sigma_{0n})$ of $T_{X_\sigma}M$ has a diameter less than $\frac{m}{1-(\frac{\Lambda}{\delta_M})^2m^2}\leq m(1+2(\frac{\Lambda}{\delta_M})^2m^2)$. Moreover, if $\sigma'$ is a $(n-1)$-face and $H$ is the iso-barycentre of $\Delta_\sigma$, then the distance from $H$ to $\sigma'$ is equal to $\frac{n\Vol \Delta_\sigma}{(n+1)\Vol\sigma'}\geq \frac{m}{C(n)\Theta^{2n}}$. 
Hence, Theorems \ref{coorbar} and \ref{Lipexp} imply that we can choose $C(n)$ large enough so that $T_\sigma$ have diameter less than $m\bigl(1+4(\frac{\Lambda}{\delta_M})^2m^2\bigr)$, the $(1+\eta)T_\sigma$ contain all the $B^\sigma_{k}(\Delta^n)$ ($0\leq k\leq n$) and $(1-\eta)T_\sigma\cap B^{\sigma'}_{k}(\Delta^n)=\emptyset$ for all $\sigma'\in S_n\setminus\{\sigma\}$ and all $0\leq k\leq n$. In Particular the $\bigl((1-\eta)T_\sigma\bigr)_{\sigma\in S_n}$ are disjoint.

Let $\tau\in S_p$ be a simplex of $T$. We now show by recurrence on $p$ that 
$$N_\tau=\bigcup_{
  \begin{matrix}
   \sigma\in N_n(\tau)\\
   0\leq k\leq p
  \end{matrix}}B_{i_\sigma^{-1}(x_{i_\tau(k)})}^\sigma(\Delta^n)$$
is a neighbourhood of $T_\tau$ in $M$ and that $d(T_\tau,\pl N_\tau)\geq \frac{m}{C(n)\Theta^{2n}}$.

Note that for any $\sigma\in S_n$, $\Delta_\sigma$ has heights greater than $\frac{n\Vol \Delta_\sigma}{\Vol\sigma'}\geq \frac{m}{C(n)\Theta^{2n}}$, where $\sigma'$ is the face of $\sigma$ with dimension $n-1$ and smallest volume. So the case $p=0$ derives from the first axiom of geodesic triangulations and from Lemma \ref{coorbar}. 

If $\tau$ is of dimension $p\geq1$ then  $\exp_{x_{i_\tau(k)}}\bigl(\{\sum_{j=0}^p\theta_jv^\sigma_{ki_\sigma^{-1}(x_{i_\tau(j)})}/\,\theta_j>0,\,\sum\theta_j<1\}\bigr)$ is interior to $\dis\bigcup_{\sigma\in St_n(x_{i_\tau(k)})}B^\sigma_{i_\sigma^{-1}(i_\tau(k))}(\Delta^n)$ for any $0\leq k\leq p$,  by the first axiom of the geodesic triangulation and the smallness of $m$. Its boundary is also included in $N_\tau$ by the recurrence assumption and Lemma \ref{coorbar}. So $T'_\tau$ is a subset of $N_\tau$. It remains to show that $d(T'_\tau,\pl N_\tau)\geq \frac{m}{C(n)\Theta^{2n}}$, which combined with Lemma \ref{coorbar} will imply that $T_\tau$ is included in $N_\tau$ and that $d(T_\tau,\pl N_\tau)\geq \frac{m}{C(n)\Theta^{2n}}$ for $C(n)$ large enough. By pulling back the vertices of $N_\tau$ to $T_{x_{i_\tau}(0)}$ under the map $\exp_{x_{i_\tau(0)}}$, we can assume that $(M^n,g)=(\R^n,eucl)$ (by Lemma \ref{coorbar} this operation does not change $m$ and $\Theta$ too much for $C(n)$ large enough). In that case $T_\tau$ is a real simplex and by convexity argument on the distance function, $d(T_\tau, \pl N_\tau)$ is bounded from below by the infimum of the distances between disjoint faces of a $n$-simplex of $T_\tau$. For a $n$-simplex with mesh $m$ and thinness $\Theta$, an easy computation, based on multi-linearity of the determinant, gives that this distance is bounded from below by $\frac{m}{C(n)\Theta^{2n}}$.

We easily infer (3) from what precedes and from the second axiom of the definition of the geodesic triangulations. We also have $\dis\bigcup_{\sigma\in S_n}\cup_{k=0}^nB^\sigma_k(\Delta^n)$ both closed and open in $M$, and so equal to $M$. This implies that the sets $\bigl((1+\eta)T_\sigma\bigr)_{\sigma\in S_n}$ cover $M$.

The property (4) is obvious when $\sigma\subset\tau$ or $\tau\subset\sigma$ and follows from (3) when $\sigma\cap\tau=\emptyset$. In particular, (4) is true when $\sigma$ or $\tau$ is a vertex. We now suppose that $\tau$ and $\sigma$ intersect and no one is a subset of the other. As for Point (3), we pull back $\sigma$ and $\tau$ in $T_zM$, where $z$ is a vertex of $\sigma\cap\tau$. By Lemma \ref{coorbar} it remains to show that (4) is satisfied in the Euclidean case. 

Let $\alpha(n)>0$ such that for any $k$-face $\sigma$ and any face $\tau$ of $\Delta^n$ we have
$$B\bigl(\sigma\setminus B(\sigma\cap\tau,\alpha^k),\alpha^{k+1}\bigr)\cap B(\tau,\alpha^{k+1})=\emptyset$$
By dilation based on a vertex of $\tau\cap\sigma$ and rate $r\leq 1$ we get
$$B\bigl(\sigma\setminus B(\sigma\cap\tau,\alpha^kr),\alpha^{k+1}r\bigr)\cap B(\tau,\alpha^{k+1}r)=\emptyset$$
Since the linear map which maps $\Delta^n$ to any $T_s$ for $s\in S_n$ is auto-adjoint with eigenvalues in $[\frac{C(n)m}{\Theta},C(n)m]$ and by Lemma \ref{coorbar}, we get point (4).
\end{proof}
Given a geodesic triangulation $T$ of $M$, we set, for any $x\in M$, $\dim_T(x)=\inf\{p\geq0/\,\exists\sigma\in S_p, d(x,T_\sigma)\leq\frac{m_T\alpha^{p+1}}{\Diam_M}m_T\}$. This is well defined by Point (2) of Corollary \ref{discret} as soon as $m_T\leq \frac{\Diam_M}{C(n)\Theta^{4n^4}\Lambda^2}$. For any simplex $\sigma\in S_p$, we set
$$\displaylines{S_\sigma=B\Bigl(T_{\sigma},\frac{m_T\alpha^{p+1}}{\Diam_M}m_T\Bigr)\setminus\bigcup_{\tau\subset\pl\sigma} B\Bigl(T_{\tau},\frac{m_T\alpha^{\dim\tau+1}}{\Diam_M}m_T\Bigr)\cr
\overline{S}_\sigma=\bigcup_{\tau\subset\sigma}S_\tau.}$$
The following properties follow readily from Lemma \ref{coorbar} and Corollary \ref{discret}. It fundamental for our application to spectral approximations. It says that, even if a geodesic triangulation is not an actual triangulation of the manifold, you can decompose the manifold into some thickening of the generalized faces $T_\sigma$.

\begin{proposition}\label{Ssigma}
Let $(M^n,g)$ be a manifold with $\Diam_M^2|\sigma|\leq\Lambda^2$ and $T$ be a geodesic triangulation of $M$ with $10m_T\leq\inf\bigl(\inj,\frac{\Diam_M}{C(n)\Theta_T^{4n^2}\Lambda(1+\Lambda)}\bigr)$. We have the following properties
\begin{enumerate}
\item  for any $\sigma\in S_p$, we have $S_\sigma=\{x\in T_\sigma/\dim_T(x)=p\}$,
\item $M$ is the disjoint union of the $(S_\sigma)_{\sigma\in K}$ and for any $(\sigma,\tau)\in K^2$, we have $\overline{S}_\sigma\cap\overline{S}_\tau=\overline{S}_{\sigma\cap\tau}$.
\item for any $\sigma\in S_p$ and $\tau\in St_n(\sigma)$ we have that $\Vol S_\sigma\leq C(n)\bigl(\frac{m_T}{\Diam_M}\bigr)^{n-p}\Vol S_\tau.$
\end{enumerate}
\end{proposition}

\subsection{Construction of good geodesic triangulations}

Given a compact manifold, we can use the Riemannian exponential maps to construct some geodesic triangulations at the neighbourhood of any point with bounded thinness and arbitrarily small mesh (image of some Euclidean triangulations of the tangent space) and then adapt the Cheeger-M\"uller-Schr\"ader (\cite{CMS}) procedure to interpolate these local triangulations in a global, controlled triangulation of the manifold. 

To make easier the control of the thinness in our construction, we will work with an alternative (fortunately equivalent in bounded geometry) thinness $\tilde{\Theta}_T$ of triangulations. In that purpose we replace the Gramm matrix $A^\sigma$ by the matrix $$\tilde{A}^\sigma=\Bigl(\frac{1}{2}\Bigl[d^2(X_\sigma,x_{i_\sigma(k)})+d^2(X_\sigma,x_{i_\sigma(l)})-d^2(x_{i_\sigma(k)},x_{i_{\sigma(l)}})\Bigr]\Bigr)_{k,l}$$ in the definition of the thinness given in the introduction. By Theorem \ref{Lipexp}, if $(M^n,g)$ satisfies $\delta_M^2|\sigma|\leq\Lambda^2$ and $m_T<\inf\bigl(\inj,\frac{\Diam_M}{2\Lambda}\bigr)$ then we have
\begin{align*}
\bigl(1-\Lambda^2(\frac{m_T}{\delta_M})^2\bigr)^2|v^\sigma_k|_{g_X}^2&\leq d(X,x_{i_\sigma(k)})^2\leq\bigl(1+\Lambda^2(\frac{m_T}{\delta_M})^2\bigr)^2|v^\sigma_{k}|_{g_X}^2,\\
\bigl(1-\Lambda^2(\frac{m_T}{\delta_M})^2\bigr)^2|v^\sigma_l|_{g_X}^2&\leq d(X,x_{i_\sigma(l)})^2\leq\bigl(1+\Lambda^2(\frac{m_T}{\delta_M})^2\bigr)^2|v^\sigma_l|_{g_X}^2,\\
\bigl(1-\Lambda^2(\frac{m_T}{\delta_M})^2\bigr)^2|v^\sigma_l-v^\sigma_k|_{g_X}^2&\leq d(x_{i_\sigma(l)},x_{i_\sigma(k)})^2\leq\bigl(1+\Lambda^2(\frac{m_T}{\delta_M})^2\bigr)^2|v^\sigma_l-v^\sigma_k|_{g_X}^2,
\end{align*}
which easily gives
$\Bigl|A^\sigma_{kl}-\tilde{A}^\sigma_{kl}\Bigr|\leqslant6m_T^2(\Lambda \frac{m_T}{\delta_M})^2$. This easily implies the existence of some functions $C_1$, $C_2$ such that $\tilde{\Theta}_{T}\leqslant C_1(\Theta_T,n)$ and $\Theta_{T}\leqslant C_2(\tilde{\Theta}_T,n)$ as soon we have	$m_T<\inf\bigl(\inj,\frac{\Diam_M}{2\Lambda}\bigr)$. Note that this two thinnesses coincide for Euclidean simplicial complexes and the our thinness is essentially the inverse of the fatness used in \cite{CMS}.

Let $(x_i)_{i\in I}$ be a maximal family of points of $M$ such that the balls $B_{x_i}(10\sqrt{\varepsilon})$ are disjoint. Let $I_1,\cdots, I_k$ be a partition of $I$ into (non empty) parts such that each $(B_{x_i}(40\sqrt{\varepsilon}))_{i\in I_j}$ is a maximal family of disjoint balls among the $(B_{x_i}(40\sqrt{\varepsilon}))_{i\in I\setminus \cup_{k<j} I_k}$. Since for $i_k\in I_k$ and for each $j<k$, $B_{x_{i_k}}(40\sqrt{\varepsilon})$  has to intersect at least one ball $B_{x_{i_j}}(40\sqrt{\varepsilon})$ with $x_{i_j}\in I_j$, the Bishop-Gromov inequality gives us
\begin{align*}
 k\Vol B_{x_{i_k}}(120\sqrt{\varepsilon})&\leqslant\sum_j\Vol B_{x_{i_j}}(160\sqrt{\varepsilon})\leqslant\sum_j\Vol B_{x_{i_j}}(10\sqrt{\varepsilon})\max_j\frac{\Vol B_{x_{i_j}}(160\sqrt{\varepsilon})}{\Vol B_{x_{i_j}}(10\sqrt{\varepsilon})}\\
 &\leqslant C(n)\Vol B_{x_{i_k}}(120\sqrt{\varepsilon})
\end{align*}
and so we have $k\leqslant C(n)$ for any $\varepsilon\leqslant\frac{c(n)}{\Lambda}$.

By iteration, we will construct a family of geodesic triangulations, possibly with boundary, $C_1,\cdots, C_k$ in $M$ with mesh less than $\varepsilon C(i,n)$, thinness $\tilde{\Theta}_{C_i}\leqslant C(i,n)$ and whose vertices of the boundary are outside the set $\cup_{j=1}^i\cup_{l\in I_j}B_{x_l}(30\sqrt{\varepsilon}-c(i,n)\varepsilon)$ for any $\varepsilon\leqslant C(\inj,n,\frac{\Diam}{\Lambda})$. For $i=k$ and $\varepsilon\leqslant C(\inj,n,\frac{\Diam}{\Lambda})$, $C_k$ will be a geodesic triangulation of $M$ (without boundary since $\cup B_{x_i}(20\sqrt{\varepsilon})=M$) with mesh less than $C(k,n)\varepsilon$ and thinness less than $\Theta(n)=C(k,n)$.

For any $\varepsilon>0$, there exists a constant $C(n)>0$ such that $\R^n$ admits a triangulation with mesh less than $\varepsilon$ and thinness less than $C(n)$. For any $i\in I$, let $T_i$ be such a triangulation of $T_{x_i}M$ and $T'_i$ the subcomplex whose simplices are those of $T_i$ contained in $B_{0_{x_i}}(30\sqrt{\varepsilon})\subset T_{x_i}M$. For any $j\in\{1,\cdots,k\}$, we set $K_j$ the simplicial complex $\cup_{i\in I_j}T'_i$. $K_j$ is naturally identified with an abstract simplicial complex of $M$ with vertices $\{\exp_{x_i}(y),y\in T_i',i\in I_j\}$. By Theorem \ref{Lipexp}, we have $m_{K_j}\leqslant\varepsilon/2$ and  $\tilde{\Theta}_{K_j}\leqslant 2C(n)$ for any $\varepsilon\leqslant C(n,\inj,\frac{\delta_M}{\Lambda})$.

We set $C_1=K_1$. Assume that $C_i$ is constructed. We now construct $C_{i+1}$ by interpolation of $C_i$ with $K_{i+1}$.
For any $l\in I_{i+1}$, we consider in $T_{x_l}M$ the complex $T'_l$ and the complex $S_l$ whose vertices are the pull back by $\exp_{x_l}$ of the vertices of $C_i$ that are contained in $B_{0_{x_l}}(40\sqrt{\varepsilon})$, and whose simplices have the same combinatorial than in $C_i$. Using Theorem \ref{Lipexp} as above we get that $S_l$ is an Euclidean complex with mesh less than $2C(i,n)\varepsilon$ and thinness less than $2C(i,n)$ for $\varepsilon\leqslant\frac{1}{\eta(i,n)}\inf(\inj,\frac{\Diam}{2\Lambda})$ for $\eta(i,n)$ large enough. In what follows, for any Euclidean complex $T$, we denote by $\mathbf{T}$ its support, i.e. the union of its simplices. Let $A_l$ (respectively $A'_l$) be the complex formed by the simplices of $T'_l$ (respectively not) contained in $\mathbf{S_l}$, and $B_l$ (respectively $B'_l$) be the complex formed by the simplices of $S_l$ (respectively not) contained in $\mathbf{T'}_l$. We set $D_l$ the set of simplices of $T_{x_l}M$ that are intersection of a simplex of $A_l$ and of a simplex of $B_l$. Let $E_l$ be the set of the simplices $\sigma$ of $A'_l$ disjoint from $\mathbf{B'_l}$ and whose intersection with $\mathbf{D_l}$ is either empty or a union of faces of $\sigma$. Similarly, we set $F_l$ the set of the simplices $\sigma$ of $B'_l$ disjoint from $\mathbf{A'_l}$ and whose intersection with $\mathbf{D_l}$ is either empty or a union of faces of $\sigma$.

$D_l$ is a polyhedral complex. We obtain a triangulation $D'_l$ of $D_l$ by first barycentric
subdivision as follows. Let $C$ be a cell of $D_l$. It is a closed, convex polyhedral cell, as an intersection $\sigma\cap\sigma'$ with $\sigma$ in $A_l$ and $\sigma'$ in $B_l$. For each face $\sigma_\alpha\subset\partial C$, we set $p_\alpha$ the isobarycentre $\sigma_\alpha$. The simplices of $D'_l$ are those spanned by all sets $p_{\alpha_1}, ...,p_{\alpha_t}$, with $\sigma_{\alpha_i}\subsetneq \sigma_{\alpha_{i+1}}$ for any $1\leqslant i\leqslant t-1$. By the proof of Lemma 6.3, p.439-440, and by Lemma 7.1 3) of \cite{CMS}, there exists some constants $f(\Theta,m,n)$, $g(\Theta,n)$ and $h(\Theta,n)$ such that, up to a move of the vertices of $T'_l$ by at most $\varepsilon 2C(i,n)f(2C(i,n),n)$, the thinness of $D'_l$ is less than $g(2C(i,n),n)$ and its mesh less than $\varepsilon h(2C(i,n),n)$. Actually, we perform this deformation of the complex $T'_l$ before the definition of the complexes $A_l$, $A'_l$, $B_l$, $B'_l$, $D_l$, $E_l$ and $F_l$. Since \cite{CMS} allows a control of the thinness of the first barycentric subdivision of any intersection cell of a simplex of (the deformation of) $T'_l$ and of a simplex of $S_l$, there is no circular definition.

To extend this triangulation to $E_l\cup F_l$, we keep unchanged the simplices of $E_l\cup F_l$ that do not intercept  $\mathbf{D'_l}$ and we subdivides all the $n$-simplices of $E_l$ (or $F_l$) with non-empty intersection with $\mathbf{D_l}$. For such a simplex $\sigma$ of $E_l$, we have $\sigma\cap  \mathbf{D_l}=\partial\sigma\cap \mathbf{\partial D_l}$, since any cell of $D_l$ is covered by some simplex of $A'_l$, and so different from $\sigma$. The triangulation $D_l'$ induces a partition of $\sigma\cap  \mathbf{D_l}$, which by definition of $E_l$ is a triangulation $(\partial\sigma)'$  of $\partial \sigma$. We then subdivides $\sigma$ by forming all the simplices spanned its barycentre and by a face of $(\partial\sigma)'$. Thus we get a triangulation $G_l$ of $E_l\cup F_l$. Once again, by Lemma 7.1 3) of \cite{CMS}, the thinness of $G_l$ is less than $g(2C(i,n),n)$ and its mesh less than $h(2C(i,n),n)$.

We set $C'_{i+1,l}=G_l\cup D'_l$ and $C(i+1,n)=\max\bigl(1,g(2C(i,n),n),h(2C(i,n),n)\bigr)$. Then $C'_{i+1,l}$ is a simplicial complex of $T_{x_l}M$ with mesh less than $\varepsilon C(i+1,n)$ and thinness less than $C(i+1,n)$. Indeed, since $G_l$ and $D'_l$ are simplicial complexes, we just need to remark that a simplex of $G_l$ and a simplex of $D'_l$ intersect along a simplex of $\partial D'_l$ by what precedes.  Moreover, any point of  $\mathbf{T'_l}\cup\mathbf{S_l}$ is either in $\mathbf{D_l}$ or contained in a $n$-simplex of $A'_l$ or of $B'_l$. So $(\mathbf{T'_l}\cup\mathbf{S_l})\setminus\mathbf{C'_{i+1,l}}$ is covered by the simplices of $A'_l$ (respectively of $B'_l$) that intersect $\mathbf{B'_l}$ (respectively $A'_l$) or whose intersection with $\mathbf{D_l}$ is not a union of their faces. In the former case, the simplex obviously intersects both $\partial S_l$ and $\partial T'_l$ and so is at distance from $\partial \mathbf{S_l}$ and from $\partial \mathbf{T'_l}$ less than $2m_{C'_{i+1,l}}$. It is the same in the latter case, since if $\sigma$ is a simplex of (for instance) $A'_l$ whose non empty intersection with $\mathbf{D_l}$ is not a union of face of $\sigma$, then a face $\sigma'$ of $\sigma$ must satisfy $\sigma'\cap\mathbf{D_l}\neq\emptyset$ and $\sigma'\setminus\mathbf{D_l}\neq\emptyset$. It first gives that $\sigma$ intersect $\mathbf{S_l}$ (but is not contained in $\mathbf{S_l}$) and so is at distance from $\partial \mathbf{S_l}$ less than $m_{C'_{i+1,l}}$. By definition of $D_l$, $\sigma'\cap\mathbf{D_l}$ is covered by some simplices of $T'_l$ contained in $\mathbf{S_l}$, and since $T'_l$ is a simplicial complex, we get that $\sigma'\subset \mathbf{S_l}$. So $\sigma'$ is covered by some simplices of $S_l$. Since $\sigma'$ is not in $D_l$, we infer that $\sigma'$ intersect some simplex of $B'_l$. Hence, $\sigma'$ is at distance from $\partial\mathbf{T'_l}$ less than $m_{C'_{i+1,l}}$ and $\sigma$ is at distance from $\partial\mathbf{T'_l}$  less than $2m_{C'_{i+1,l}}$. We infer that $(\mathbf{T'_l}\cup\mathbf{S_l})\setminus\mathbf{C'_{i+1,l}}$ is covered by the intersection of the $2m_{C'_{i+1,l}}$-tubular neighbourhoods of $\partial T'_l$ and $\partial S_l$.

We now set $C_{i+1,l}$ the union of the vertices of $C_i\setminus B_{x_l}(35\sqrt{\varepsilon})$ and of the image by $\exp_{x_l}$ of the vertices of $G_l\cup D'_l$. We endow it with the abstract structure of complex obtained by gathering that of $C_i\setminus B_{x_l}(30\sqrt{\varepsilon})$ and that of $G_l\cup D'_l$. Since the complex $S_l$ is not deformed during the previous interpolation and since the only vertices of $S_l$ that disappears during the interpolation are in $B_{x_l}(35\sqrt{\varepsilon})$, we really get an abstract structure of simplicial complex on $C_{i+1,l}$ such that any simplex is contained in a $n$-dimensional simplex. Moreover, by Theorem \ref{Lipexp}, $C_{i+1,l}$ is a triangulation with boundary of $(M^n,g)$ with thinness less than $2C(i+1,n)$ and mesh less than $2C(i+1,n)\varepsilon$ for any $\varepsilon\leqslant C(n,i_M,\frac{\Diam}{\Lambda})$. Finally, by what precedes the vertices of the boundary of $C_{i+1,l}$ are the same as the vertices of the boundary of $C_i$ outside $B_{x_l}(30\sqrt{\varepsilon}+C(i,n)\varepsilon)$ and are at distance less than $C(i,n)\varepsilon$ from $\partial B_{x_l}(30\sqrt{\varepsilon})$	inside $B_{x_l}(30\sqrt{\varepsilon}+C(i,n)\varepsilon)$ for any $\varepsilon\leqslant C(n,i_M,\frac{\Diam}{\Lambda})$ (once again by Theorem \ref{Lipexp}). From this, we get that the vertices of the boundary of $C_{i+1,l}$
are outside $\bigl(\cup_{j=1}^i\cup_{p\in I_j}B_{x_p}(30\sqrt{\varepsilon}-c(i+1,n)\varepsilon)\bigr)\cup B_{x_l}(30\sqrt{\varepsilon}-c(i+1,n)\varepsilon)$.

$C_{i+1,l}$ is just the interpolation between $C_i$ and $T'_l$. But since the family of balls $B_{x_l}(40\sqrt{\varepsilon})$, $l\in I_{i+1}$ are disjoint, we can interpolate $C_i$ with all the $T'_l$ ($l\in I_{i+1}$) simultaneously to get $C_{i+1}$. Note that the constant $c(i+1,n)$ and $C(i+1,n)$ will be the same whatever the cardinal of $I_{i+1}$ is since the operations done during the interpolation of two different $T'_l$ do not interact. So we get the geodesic triangulation $C_{i+1}$ with all the needed estimates.
\bigskip

Note that the image of any simplex of our geodesic triangulation by the barycentric coordinates map associated to its vertices gives an embedded  simplex of $M$ and thus a true triangulation of $M$ whose edges are minimizing geodesic segments.

\section{Estimates on the eigenfunctions}\label{Estimeigenf}

The following proposition gives bounds on the gradient and Hessian of the eigenfunctions.

\begin{proposition}\label{estimapprox}
 Let $(M,g)$ be a compact Riemannian manifold with $delta_M^2|\sigma|\leq\Lambda^2$. For any  $f\in E_p$, we have that 
\begin{align}
\|f\|_\infty^{~}&\leq C(n)e^{\frac{n^2}{2}\Lambda}(1+\Diam_M^2\lambda_p)^{\frac{n}{4}}\|f\|_2^{~},\\
\Diam_M\|df\|_\infty^{~}&\leq C(n)e^{\frac{n^2}{2}\Lambda}(1+\Diam_M^2\lambda_p)^{\frac{n}{4}}\Diam_M\|df\|_2^{~}\leq C(n)e^{\frac{n^2}{2}\Lambda}(1+\Diam_M^2\lambda_p)^{\frac{n+2}{4}}\|f\|_2^{~},\\
\Diam_M\|Ddf\|_\infty^{~}&\leq C(n)e^{n^3\Lambda}(1+\Diam_M^2\lambda_p)^\frac{n^2}{2}\|df\|_\infty^{~}\leq C(n)e^{2n^3\Lambda}(1+\Diam_M^2\lambda_p)^{\frac{3n^2}{4}}\|df\|_2^{~}.
\end{align}
\end{proposition}

\begin{remark}
  The three first inequalities of Proposition \ref{estimapprox} are valid under the weaker assumption $\Diam_M^2\Ric\geq -\Lambda^2g$.
\end{remark}

\begin{proof}
  The proof of Proposition \ref{estimapprox} is based on a Moser iteration schema. For any $f\in H^{1}(M),$ we have the Sobolev inequality $\|f\|_\frac{2n}{n-2}\leq C(n)e^{\frac{n}{2}\Lambda}\Diam_M\|df\|_2+\|f\|_2$ (if $M$ is a surface, this inequality, and what follows, is valid with $n=4$).

Let $\overline{\Delta}=D^*D$ be the rough Laplace operator on covariant tensors of $(M^n,g)$ (where $D^*$ is the $L^2$ adjoint of $D$). The rough Laplace operator coincides on the functions with the usual Laplace operator and, on the 1-forms it is related to the Hodge Laplace operator by the Bochner formula $\Delta=\overline{\Delta}+\Ric$. For any tensor $T$ on $M$ it satisfies the equality $\<\overline{\Delta}T,T\>=|D T|^2+\frac{1}{2}\Delta\bigl(|T|^2\bigr)$.

Let $T$ be any tensor on $M$ and $V$ be a field of symmetric endomorphisms on the covariant tensor bundle of $M$. Fix a real $\underline{V}\geq0$ such that ${<}V(T),T{>}\geq-\underline{V}|T|^2$ for all tensors $T$.
We set $u=\sqrt{|T|^2+\epsilon^2}$, then we have that
$$\displaylines{
u\Delta u=\frac{1}{2}\Delta(u^2)+|du|^2\leq\frac{1}{2}\Delta(|T|^2)+|D T|^2=\<\overline{\Delta}T,T\>\leq|(\overline{\Delta}+V)T|u+\underline{V}u^2.}
$$
This inequality and the Green formula gives, for any real $k>1/2$
$$\displaylines{
\|d(u^k)\|_2^2=\frac{k^2}{2k-1}\int_M\frac{\<du,d(u^{2k-1})\>}{\Vol M}=\frac{k^2}{2k-1}\int_M\frac{(u\Delta u)u^{2k-2}}{\Vol M}\cr
\leq\frac{k^2}{2k-1}\left[\|(\overline{\Delta}+V)T\|_{2k}\|u\|^{2k-1}_{2k}+\underline{V}\|u\|^{2k}_{2k}\right]
}$$
We apply the above Sobolev inequality to the function $u^k$ and make then $\epsilon$ tends to $0$. This gives us the following inequality
\begin{equation}
  \label{*}
  \|T\|_{\frac{2kn}{n-2}}^k\leq\Bigl(\|T\|_{2k}^k+\frac{C(n)k\Diam_Me^{\frac{n}{2}\Lambda}}{\sqrt{2k-1}}\sqrt{\|(\overline{\Delta}+V)T\|_{2k}\|T\|_{2k}^{2k-1}+\underline{V}\|T\|_{2k}^{2k}}\Bigr)
\end{equation}
We have $E_p={\rm Vect}\{f_i,\,i\leq p\}$. For any $k\geq1$, we set $\dis A_k=\sup_{f\in E_p\setminus\{0\}}\frac{\|f\|_k}{\|f\|_2}$ (resp. $\dis B_k=\sup_{f\in E_p\setminus\{0\}}\frac{\|df\|_k}{\|df\|_2}$). Since $E_p$ is stable by $\Delta$ we have
$$\dsl{
\hfill\|\Delta f\|_{2k}\leq A_{2k}\|\Delta f\|_2\leq A_{2k}\lambda_p\|f\|_2,\hfill
\|\Delta df\|_{2k}\leq B_{2k}\|\Delta df\|_2\leq B_{2k}\lambda_p\|df\|_2.\hfill}
$$
Hence, by applying inequality $\ref{*}$ to $T=f$ and $V=0$ (resp. to $T=df$ and $V=\Ric$), for any $f\in E_p$, we get
$$\displaylines{A_{\frac{2nk}{n-2}}\leq\left(1+\frac{C(n)e^{\frac{n}{2}\Lambda}k\Diam_M\sqrt{\lambda_p}}{\sqrt{2k-1}}\right)^{1/k}A_{2k}\cr
B_{\frac{2kn}{n-2}}\leq\left(1+\frac{C(n)e^{\frac{n}{2}\Lambda}k}{\sqrt{2k-1}}\sqrt{(n-1)\Lambda^2+\Diam_M^2\lambda_p}\right)^{1/k}B_{2k}}$$
We multiply the inequalities obtained by setting successively $k=\nu^j$ with $\nu=\frac{n}{n-2}>1$ and $j\in\mathbb{N}$. Since $A_m$ tends to $A_\infty$ (resp. $B_m$ tends to $B_\infty$) when $m$ tends to $\infty$, we get
$$\max(A_\infty,B_\infty)\leq\prod_{j=0}^{\infty}\left(1+\frac{C(n)e^{\frac{n}{2}\Lambda}\nu^j}{\sqrt{2\nu^j-1}}\sqrt{(n-1)\Lambda^2+\Diam_M^2\lambda_p}\right)^{\frac{1}{\nu^j}}.$$
To get a more convenient upper bound, note that
$
1+a\sqrt{b}\leq\sqrt{1+b}\,(1+a)$
and that the infinite product $\prod_{j=0}^{\infty}\left(1+\frac{\nu^j}{\sqrt{2\nu^j-1}}\right)^{\frac{1}{\nu^j}}$ converges, hence
$$\max(A_\infty,B_\infty)\leq C'(n)e^{\frac{n^2}{4}\Lambda}\bigl(1+(n-1)\Lambda^2+\Diam_M^2\lambda_p\bigr)^{\frac{n}{4}}.$$
This gives us the first three inequalities of proposition \ref{estimapprox}. For any $f\in E_p$, we have (see for instance \cite{Au5})
\begin{equation}
  \label{Defcom}
  \<\overline{\Delta} D df,D df\>\leq\<(D^*R)df,D df\>+\frac{C(n)\Lambda^2}{\Diam_M^2}|D df|^2+\<D\overline{\Delta}df,D df\>.
\end{equation}

Now, if we set $u=\sqrt{|D df|^2+\epsilon^2}$, we have $u\Delta u\leq \<\overline{\Delta} D df,D df\>$. From Lemma \ref{Defcom}, we infer
\begin{eqnarray}\label{**}
    \int_M&&\hskip-6mm|d(u^k)|^2\leq\frac{k^2}{2k-1}\int_M(u\Delta u)u^{2(k-1)}\nonumber\\
&&\leq\frac{k^2}{2k-1}\Bigl(\frac{C(n)\Lambda^2}{\Diam_M^2}\|u\|_{2k}^{2k}+\int_M\<D\overline{\Delta}df,D df\>u^{2(k-1)}+\int_M\<D^*R df,D df\>u^{2(k-1)}\Bigr)
\end{eqnarray}

The divergence theorem applied to the field $u^{2(k-1)}\<\overline{\Delta}df,D_{{\bullet}}df\>^{\#}$, gives ($\forall k\geq1$)
$$\displaylines{
\int_M\<D\overline{\Delta}df,D df\>u^{2(k-1)}=\int_M|\overline{\Delta}df|^2u^{2(k-1)}-2(k-1)\sum_i\<\overline{\Delta}df,D df(e_i)\>du(e_i)u^{2k-3}\cr
\leq\int_M|\overline{\Delta}df|^2u^{2(k-1)}+2(k-1)\int_M|\overline{\Delta}df||du|u^{2(k-1)}\cr
\leq\frac{k-1}{2}\int_M|du|^2u^{2(k-1)}+(2k-1)\int_M|\overline{\Delta}df|^2u^{2(k-1)}\cr
\leq\frac{k-1}{2}\int_M|du|^2u^{2(k-1)}+2(2k-1)\int_M(|\Delta df|^2+|\Ric(df)|^2)u^{2(k-1)}\cr
\leq\frac{k-1}{2}\int_M|du|^2u^{2(k-1)}+2(2k-1)(\|\Delta df\|_\infty^2+\frac{C(n)\Lambda^4}{\Diam_M^4}\|df\|_\infty^2)\int_Mu^{2(k-1)}.}$$
The same process applied to $u^{2(k-1)}(tr_{1,3}^{~}(\<R_{(\bullet,\bullet)}df,D_\bullet D df\>))^\#$, combined to the equality $$\sum_{i,j}\<Rdf(e_i,e_j),D^2df(e_i,e_j)\>=\frac{1}{2}|R df|^2,$$ gives:
$$\displaylines{
\int_M\<D^*R df,D df\>u^{2(k-1)}\hfill\cr
=\int_M\frac{1}{2}|R df|^2u^{2(k-1)}+2(k-1)\sum_{i,j}\int_M\<R^E(e_i,e_j)df,D_{e_j}df\>du(e_i)u^{2k-3}\cr
\leq\frac{k-1}{2}\int_M|du|^2u^{2(k-1)}+(2k-1)\int_M|R df|^2u^{2(k-1)}\cr
\leq\frac{k-1}{2}\int_M|du|^2u^{2(k-1)}+(2k-1)\frac{C(n)\Lambda^4}{\Diam_M^4}\|df\|^2_\infty\int_Mu^{2(k-1)}}
$$
Since $\int_M|du|^2u^{2(k-1)}=\frac{1}{k^2}\int_M|d(u^k)|^2,$ inequality $(\ref{**})$ implies:
$$\displaylines{\|d(u^k)\|_2^2\leq 4k^2\Bigl(\frac{C(n)\Lambda^2}{\Diam_M^2}\|u\|_{2k}^{2k}+\frac{C(n)\Lambda^4}{\Diam_M^4}\|df\|_\infty^2\|u\|_{2k-2}^{2k-2}+\|\Delta df\|_\infty^2\|u\|_{2k-2}^{2k-2}\Bigr)\cr
}$$
Now, since $\|\Delta df\|_\infty\leq B_\infty\|\Delta df\|_2\leq B_\infty\lambda_p\|df\|_\infty$, we get
$$\displaylines{\|d(u^k)\|_2^2\leq k^2\frac{C(n)e^{\Lambda n^2}}{\Diam_M^2}\|u\|_{2k-2}^{2k-2}\Bigl(\|u\|_{\infty}^{2}+\frac{\|df\|_\infty^2}{\Diam_M^2}(1+\Diam_M^2\lambda_p)^\frac{n+4}{2}\Bigr).
}$$
We can now apply the Sobolev inequality to get:
$$\displaylines{\|D df\|_{\frac{2kn}{n-2}}^k\leq\|D df\|_{2k}^k+ kC(n)e^{\Lambda n^2}\|D df\|_{2k-2}^{k-1}\Bigl(\|D df\|_{\infty}+\frac{\|df\|_\infty}{\Diam_M}(1+\Diam_M^2\lambda_p)^{1+\frac{n}{4}}\Bigr)\cr
\leq\|D df\|_{2k-2}^{k-1}\|D df\|_\infty\Bigl(1+C(n)e^{\Lambda n^2}k\bigl(1+(1+\Diam_M^2\lambda_p)^{\frac{n+4}{4}}\frac{\|df\|_\infty}{\Diam_M\|D df\|_\infty}\bigr)\Bigr)\cr
\leq\|D df\|_{2k-2}^{k-1}\|D df\|_\infty\Bigl(1+C(n)e^{\Lambda n^2}k(1+\Diam_M^2\lambda_p)^{\frac{n+4}{4}}\Bigr)}$$
since we can suppose that $\Diam_M\|D df\|_\infty\geq\|df\|_\infty$. Hence (set $k=a_l/2+1$)
$$F_{l+1}\leq\Bigl[1+kC(n)e^{\Lambda n^2}(1+\Diam_M^2\lambda_p)^{\frac{n+4}{4}}\Bigr]^{\frac{2}{\nu}}F_l,
$$
where $\nu=\frac{n}{n-2}>1$, $F_l=\Bigl(\frac{\|D df\|_{a_l}}{\|D df\|_\infty}\Bigr)^\frac{a_l}{\nu^l}$ and $(a_l)$ is the sequence defined by $a_0=2$ and $a_{l+1}=\frac{n}{n-2}(a_l+2)$. Since the sequence $a_l/\nu^l$ tends to $n$, we get
$$\displaylines{\frac{\|D df\|_\infty}{\|D df\|_2}\leq\prod_{i=0}^\infty\left(1+C(n)e^{\Lambda n^2}a_i(1+\Diam_M^2\lambda_p)^\frac{n+4}{4}\right)^{1/\nu^i}.}$$
The previous inequality gives
$$\|D df\|_\infty\leq C(n)e^{\frac{n^3}{2}\Lambda}(1+\Diam_M^2\lambda_p)^{\frac{(n+4)n}{8}}\|D df\|_{2}$$ 
But if we integrate the Bochner formula  $\<\Delta df,df\>=\frac{1}{2}\Delta|df|^2+|D df|^2+\Ric(df,df)$ we easily get
$$\Diam_M\|D df\|_2\leq \sqrt{\Diam_M^2\lambda_p+(n-1)\Lambda}\|df\|_2$$
So we have
$$\Diam_M\|D df\|_\infty\leq C(n)e^{n^3\Lambda}(1+\Diam_M^2\lambda_p)^\frac{(n+2)^2}{8}\|df\|_{\infty}.$$
\end{proof}

Proposition \ref{estimapprox} implies that at small scale, the eigenfunctions are almost affine.

\begin{lemma}\label{Disceigenf}
Let $(M,g)$ be a compact $n$-manifold which satisfies $\delta^2|\sigma|\leq\Lambda^2$, and $T$ be a geodesic triangulation of $M$ such that $10m_T\leq\inf(\inj,\frac{\Diam_M}{2\Lambda})$.
For any $\sigma\in S_n$, we define a function $L_\sigma^f$ on $10\cdot T_\sigma$ by
\begin{align*}
L_\sigma^f\bigl(\exp_{X_\sigma}(\sum_{j=1}^n\theta_j v_j^\sigma)\bigr)=f(X_\sigma)+\sum_{j=1}^n\theta_j\bigl[f(x_{i_\sigma(j)})-f(X_\sigma)\bigr].\end{align*}

Then we have the following estimates on $10\cdot T_\sigma$
\begin{align*}
\|f-L_\sigma^f\|_\infty&\leq C(n)e^{2n^3\Lambda}(1+\Diam_M^2\lambda_p)^\frac{(n+1)^2}{2}\bigl(\frac{m_T}{\Diam_M}\bigr)^2\|f\|_2\\
\|df-dL_\sigma^f\|_\infty&\leq C(n)\Theta^{2}e^{2n^3\Lambda}(1+\delta^2\lambda_p)^\frac{3n^2}{4}\frac{m_T}{\delta_M}\|df\|_2.
\end{align*}
\end{lemma}

\begin{proof}
  We set $v=\sum_{j=1}^n\theta_j v_j^\sigma$ and $\gamma(t)=\exp_{X_\sigma}(tv)$, then we have that
$$\displaylines{
    \Bigl|f\bigl(\gamma(1)\bigr)-f(X_\sigma)-g_x\bigl(D f(x),\dot{\gamma}(0)\bigr)\Bigr|=\Bigl|\int_0^1\int_0^s\frac{d^2}{dt^2}f\circ\gamma(t)\,dt\,ds\Bigr|\hfill\cr
\hfill\leq\Bigl|\int_0^1\int_0^s|D df|\circ\gamma(t)\,dt\,ds\Bigr||\dot{\gamma}(0)|^2\leq\|D df\|_\infty\frac{d\bigl(X_\sigma,\gamma(1)\bigr)^2}{2}.}
$$
This implies that
$$\displaylines{\Bigl|f\bigl(\exp_{X_\sigma}(\sum_{j=1}^n\theta_j v_j^\sigma)\bigr)-f(X_\sigma)-\sum_{j=1}^n\theta_jg_x(D f(x), v_j^\sigma)\Bigr|\hfill\cr
\hfill\leq C(n)e^{2n^3\Lambda}(1+\Diam_M^2\lambda_p)^\frac{(n+1)^2}{2}\bigl(\frac{m_T}{\Diam_M}\bigr)^2\|f\|_2.}$$
If we set $\theta_j=\delta_{jk}$, then this inequality gives that
\begin{equation}\label{ineq3}
\Bigl|f(x_{i_\sigma(k)})-f(X_\sigma)-g_x(D f(x), v_k^\sigma)\Bigr|\leq C(n)e^{2n^3\Lambda}(1+\Diam_M^2\lambda_p)^\frac{(n+1)^2}{2}\bigl(\frac{m_T}{\Diam_M}\bigr)^2\|f\|_2,
\end{equation}
which combined with the previous inequality gives the first result.

Set $L$ the linear form on $T_{X_\sigma}M$ such that $L_\sigma^f=L\circ\exp_{X_\sigma}^{-1}$. For any $w\in T_{X_\sigma}M$ we set $w(t)$ the parallel transport of $w$ along $\gamma$. Then by Theorem \ref{BuserKarcher}, we have that
\begin{equation}
  \label{ineq2}
  \Bigl|dL_\sigma^f\bigl(w(1)\bigr)-L(w)\Bigr|\leq\|L\|\bigl|d\exp_{X_\sigma}^{-1}\bigl(w(1)\bigr)-w\bigr|
\leq4\|L\||w|(\frac{\Lambda m}{\delta})^2.
\end{equation}
On the other hand, we have that
$$\displaylines{\bigl|df\bigl(w(1)\bigr)-df(w)\bigr|\leq\Bigl|\int_0^1\frac{d}{dt}df\bigl(w(t)\bigr)dt\Bigr|\leq\Bigl|\int_0^1Ddf\bigl(\dot{\gamma},w(t)\bigr)dt\Bigr|\leq\|Ddf\|_\infty|v||w|.}$$
Hence we get
\begin{equation}\label{ineq1}
  \bigl|\bigl(dL_\sigma^f-df\bigr)w(1)\bigr|\leq \Bigl(2|df-L|+ C(n)e^{n^3\Lambda}(1+\delta^2\lambda_p)^\frac{n^2}{2}\|df\|_\infty\frac{m}{\delta}\Bigr)|w|.
\end{equation}
Now, by Inequality \eqref{ineq3} we have that
$$\displaylines{\bigl|df(v_j^\sigma)-L(v_j^\sigma)\bigr|=\bigl|df(v_j^\sigma)-f(x_{i_\sigma(j)})+f(X_\sigma)\bigr|\leq C(n)\Theta e^{n^3\Lambda}(1+\Diam_M^2\lambda_p)^\frac{n^2}{2}\frac{m^2}{\Diam}\|df\|_\infty,}$$
and so
$$\displaylines{|df-L|^2=\sum_{jk}(A_\sigma^{-1})_{jk}(df-L)(v_j^\sigma)(df-L)(v_k^\sigma)\leq\|A_\sigma^{-1}\|\sum_j[(df-L)(v_j^\sigma)]^2\cr
\leq \Theta^{4}C(n) e^{2n^3\Lambda}(1+\Diam_M^2\lambda_p)^{n^2}(\frac{m}{\Diam})^2\|df\|^2_\infty.}$$
If we combine this inequality with Inequality \eqref{ineq1}, we get
$$|dL_\sigma^f-df|\leq C(n)\Theta^{2}e^{n^3\Lambda}(1+\delta^2\lambda_p)^\frac{n^2}{2}\frac{m}{\delta}\|df\|_\infty,$$
which gives the result by Proposition \ref{estimapprox}.
\end{proof}

\section{Proof of Theorem \ref{approvalp}}

\subsection{A discrete $L^2$ norm}\label{DiscrL2}

We prove that $|R(f)|^2_T$ gives an approximation of $\|f\|_2^2$ on $E_p$.

\begin{proposition}\label{DiscL2} Let $(M^n,g)$ be a compact Riemannian manifold with $\Diam_M^2|\sigma|\leq\Lambda^2$ and $T$ a geodesic triangulation with $10m_T\leq\inf\bigl(\inj,\frac{\Diam_M}{C(n)\Lambda\Theta^{n}_T}\bigr).$
Then, for all $f\in E_p$, we have
  $$\Bigl|\int_Mf^2-|R(f)|^2_T\Bigr|\leq C(n)\Theta^{2n}e^{5n^3\Lambda}\bigl(1+\Diam_M^2\lambda_p\bigr)^{3n^2}\bigl(\frac{m_T}{\Diam_M}\bigr)^2\int_Mf^2.$$
\end{proposition}

\begin{proof}
Corollary \ref{discret} (2) implies that 
$$\dis\sum_{\sigma\in S_n}\int_{(1-\eta)T_\sigma}f^2\leq\int_Mf^2\leq\sum_{\sigma\in S_n}\int_{(1+\eta)T_\sigma}f^2.$$
We have $\dis\int_{(1+\eta)T_\sigma}f^2\leq\int_{(1+\eta)T_\sigma}(L_\sigma^f)^2+(2\|f\|_\infty+\|L_\sigma^f-f\|_\infty)\|L_\sigma^f-f\|_\infty$.
By Lemma \ref{Disceigenf} and Proposition \ref{estimapprox}, we have
$$\int_{(1+\eta)T_\sigma}f^2\leq\int_{(1+\eta)T_\sigma}(L_\sigma^f)^2+C(n)e^{5n^3\Lambda}(1+\Diam_M^2\lambda_p)^{3n^2}(\frac{m}{\delta})^2\|f\|_2^2\Vol \bigl((1+\eta)\cdot T_\sigma\bigr)$$
By Theorem \ref{Lipexp} we have that
\begin{align*}
&\int_{(1+\eta)T_\sigma}(L_\sigma^f)^2=\int_{(1+\eta)\Delta_\sigma}\Bigl(f(X_\sigma)+\sum_{j=1}^n\theta_j\bigl[f(x_{i_\sigma(j)})-f(X_\sigma)\bigr]\Bigr)^2dv_{\exp_{X_\sigma}^*g}\\
&\leq(1+4\Lambda^2(\frac{m}{\delta})^2)^n\int_{(1+\eta)\Delta_\sigma}\Bigl(f(X_\sigma)+\sum_{j=1}^n\theta_j\bigl[f(x_{i_\sigma(j)})-f(X_\sigma)\bigr]\Bigr)^2dv_{g_{X_\sigma}}\\
&\leq\Bigl[\int_{\Delta_\sigma}\Bigl(f(X_\sigma)+\sum_{j=1}^n\theta_j\bigl[f(x_{i_\sigma(j)})-f(X_\sigma)\bigr]\Bigr)^2dv_{g_{X_\sigma}}+2^n4\|f\|^2_\infty\bigl(\Lambda^2(\frac{m}{\delta})^2+\eta\bigr)\Vol \Delta_\sigma\Bigr]\\
&\leq\frac{2\sqrt{\det A_\sigma}}{(n+2)!}\sum_{ 0\leq j\leq j'\leq n}\hskip-4mmf(x_{i_\sigma(j)})f(x_{i_\sigma(j')})+C(n)\Theta^{2n}(\frac{\Lambda m}{\delta})^2\|f\|^2_\infty\Vol(1-\eta)T_\sigma
\end{align*}
So, we have
$$\displaylines{\int_{(1+\eta)T_\sigma}f^2\leq\frac{2\sqrt{\det A_\sigma}}{(n+2)!}\sum_{ 0\leq j\leq j'\leq n}\hskip-4mmf(x_{i_\sigma(j)})f(x_{i_\sigma(j')})\hfill\cr
\hfill+ C(n)e^{5n^3\Lambda}(1+\Diam_M^2\lambda_p)^{3n^2}\Theta^{2n}(\frac{m}{\delta})^2\|f\|_2^2\Vol \bigl((1-\eta)\cdot T_\sigma\bigr)}$$
By summing on $\sigma\in S_n$, we get $\int_Mf^2-|R(f)|^2\leq C(n)e^{5n^3\Lambda}(1+\Diam_M^2\lambda_p)^{3n^2}(\frac{m}{\delta})^2\Theta^{2n}\int_Mf^2$.
By the same way, we get the reverse inequality.
\end{proof}

\subsection{A discrete Dirichlet energy}\label{Diri}
We prove that $q_T\bigl(R(f)\bigr)$ approximates $\|df\|_2^2$ on $E_p$.

\begin{proposition}\label{Diri1}
 Let $(M^n,g)$ be a compact Riemannian manifold with $\Diam_M^2|\sigma|\leq\Lambda^2$ and $T$ a geodesic triangulation of $M$ with $10m_T\leq\inf\bigl(\inj,\frac{\Diam_M}{C(n)\Theta^{n}\Lambda}\bigr)$.
Then for any $f\in E_p$, we have that
$$\displaylines{\Bigl|q_T\bigl(R(f)\bigr)-\int_M|df|^2\Bigr|\leq  C(n)\Theta^{2n+4}e^{3n^3\Lambda}(1+\delta^2\lambda_p)^{2n^2}\frac{m_T}{\Diam_M}\int_M|df|^2.}$$
\end{proposition}

\begin{proof}
As in the proof of Proposition \ref{DiscL2}, we have
$$\int_{(1+\eta)T_\sigma}|df|^2\leq\int_{(1+\eta)T_\sigma}|dL_\sigma^f|^2+C(n)\Theta^4e^{2n^3\Lambda}(1+\Diam_M^2\lambda_p)^{n^2}\frac{m}{\Diam_M}\|df\|_\infty^2\Vol\bigl( (1+\eta)\cdot T_\sigma\bigr).$$
Using Lemma \ref{Disceigenf} and the fact that $\int_{\Delta_\sigma}|L|^2=\frac{\sqrt{\det A_\sigma}}{n!}\sum_{k,l=1}^nA_\sigma^{kl}L(v_\sigma^k)L(v_\sigma^l)$, where $L$ is the linear map defined in the proof of Proposition \ref{DiscL2} (i.e. $L(v_\sigma^k)=f(x_{i_\sigma(k)})-f(X_\sigma)$) we get the result.
\end{proof}

\subsection{A Withney map}\label{triang}

We construct an extending (i.e. Withney type) map which to each $(y_i)\in \R^N$ associates a function $f:M\to \R$ such that $f(x_i)=y_i$ for all $1\leq i\leq N$. This function $f$ has to be such that $\int f^2$ and $\int |df|^2$ be close to $|y_i|_T^2$ and $q_T(y_i)$. In that purpose we take $f$ almost linear by part.

We need first some controlled partitions of the unity on $M$ associated to the geodesic triangulations.

\begin{lemma}\label{partunit}
  Let $(M^n,g)$ be a compact, Riemannian $n$-manifold with $\delta(M)^2|\sigma|\leq\Lambda^2$ and $T$ be a geodesic triangulation of $M$ with $10m_T\leq\bigl(\inj,\frac{\Diam_M}{C(n)\Lambda^2\Theta^{4n^2}_T}\bigr)$.
There exists some smooth functions $(\phi_\sigma)_{\sigma\in S_n}$ such that
\begin{enumerate}
\item $\phi_\sigma:M\to[0,1]$, $\phi_\sigma=1$ on $S_\sigma$ and ${\rm Supp}\,\phi_\sigma\subset\overline{S}_\sigma$ for any $\sigma\in S_n$,
\item $\sum_{\sigma\in S_n}\phi_\sigma=1$ and $|d\phi_\sigma|\leq\frac{C(n)\Theta_T^{5n^2}\Diam_M}{m^2_T}$ for any $\sigma\in S_n$.
\end{enumerate}
\end{lemma}

\begin{proof}
 We set $\zeta=\frac{\alpha^{n+1}m_T}{2\Diam_M}$. Since the distance in $(T_{X_\sigma}M,g_{X_\sigma})$ between $(1+\zeta)\Delta_\sigma$ and $T_{X_\sigma}M\setminus(1+2\zeta)\Delta_\sigma$ is larger than $\frac{\zeta m}{n\Theta^{2n}}$ (see the proof of Corollary \ref{discret}) and since $\exp_{X_\sigma}$ is a $1+\Lambda^2(\frac{m}{\delta})^2$-Lipschitzian map, there exists a function $\psi_\sigma:M\to[0,1]$ such that $\psi_\sigma=1$ on $(1+\zeta)T_\sigma$, $\psi_\sigma=0$ outside $(1+2\zeta)T_\sigma$ and $|d\psi_\sigma|\leq\frac{4n\Theta^{2n}}{\zeta m}$. We set $\phi_\sigma=\frac{\psi_\sigma}{\sum_{\tau\in S_n}\psi_\tau}$, which is well defined since by Corollary \ref{discret} we have $\cup_{\tau\in S_n}(1+\zeta)T_\tau=M$. By the same kind of arguments as in the proof of Corollary \ref{discret}, we have $(1-4\zeta)T_\sigma\cap(1+2\zeta)T_\tau=\emptyset$ for $\tau\neq\sigma$, and so $\phi_\sigma=\psi_\sigma=1$ on $(1-4\zeta)T_\sigma$. We have obviously $\phi_\sigma=0$ outside $(1+2\zeta)T_\sigma$. By a volume argument, the number of non zero term in the sum $\sum_\tau d_x\psi_\tau$ is bounded from above by $2^n\Theta^{2n}$, and so we have that
$$|d\phi_\sigma|=\Bigl|\frac{d\psi_\sigma}{\sum_\tau \psi_\tau}-\frac{\psi_\sigma\sum_\tau d\psi_\tau}{(\sum_\tau\psi_\tau)^2}\Bigr|\leq\frac{5n2^n\Theta^{4n}}{\zeta m}.$$ 
\end{proof}

\begin{definition}
  Let $f_T=(y_i)\in\R^N$. For any $\sigma\in K$, we set $L^\sigma:(\theta_k)\in F\mapsto\sum_k\theta_ky_{i_\sigma(k)}$. For $m_T$ small enough, the function $L_\sigma^{f_T}= L^\sigma\circ (B_0^\sigma)^{-1}$ is defined on $S_\sigma$ and $\dis W(f_T)=\sum_{\sigma\in S_n}\phi_\sigma L_\sigma^{f_T}$ is well defined on $M$.
\end{definition}
 This extending map $W$ satisfies the following properties.
\begin{proposition}\label{EstimW}
  If $10m_T\leq\inf\bigl(\inj,\frac{\Diam_M}{C(n)\Lambda(1+\Lambda)\Theta_T^{4n^2}}\bigr)$ then
  \begin{enumerate}
  \item $R\circ W=Id_{\R^N}$,
\item for any $f_T\in\R^N$, we have that 
$$\bigl|\int_MW(f_T)^2-|(f_T)|_T^2\bigr|\leq C(n)(1+\Lambda)\Theta^{2n}(\frac{m}{\delta})^2\bigl(|f_T|_T^2+\Diam_M^2q_T(f_T)\bigr),$$
\item for any $(f_T)\in\R^N$, we have that
$$\bigl|\int_M|dW(f_T)|^2-q_T(f_T)\bigr|\leq C(n)\Theta^{20n^2}\frac{m}{\delta} q_T(f_T).$$
  \end{enumerate}
\end{proposition}

\begin{proof}
  Point $(1)$ is obvious.
We now prove point $(3)$. By Inequality \eqref{ineq2}, we have $\bigl||dL_\sigma^{f_T}|^2-|L_\sigma|^2\bigr|\leq16|L_\sigma|^2\bigl(\frac{\Lambda m}{\delta}\bigr)^2$.
Since $$|L_\sigma|^2=\sum_{k,l=1}^n(A_\sigma)^{kl}(y_{i_\sigma(k)}-y_{i_\sigma(0)})(y_{i_\sigma(l)}-y_{i_\sigma(0)}),$$ we have
$$\displaylines{\int_M|dW(f_T)|^2\geq\sum_{\sigma\in S_n}\int_{S_\sigma}|dW(f_T)|^2=\sum_{\sigma\in S_n}\int_{S_\sigma}|dL_\sigma^{f_T}|^2\hfill\cr
\geq\sum_{\sigma\in S_n}\Vol T_\sigma\bigl(1-16(\frac{\Lambda m}{\delta})^2\bigr)^{n+1}\bigl(1- \frac{m\alpha^n}{\Diam_M}\bigr)^n |L_\sigma|^2\geq\bigl(1-C(n)\frac{m}{\delta}\bigr)q_T(f_T).}$$
On the other hand, we have that
$$\displaylines{\int_M|dW(f_T)|^2=\sum_{\sigma\in S_n}\int_{S_\sigma}|dW(f_T)|^2+\sum_{p=0}^{n-1}\sum_{\tau\in S_p}\int_{S_\tau}|dW(f_T)|^2\hfill\cr
\leq \bigl(1+C(n)\frac{ m}{\delta}\bigr)q_T(f_T)+\sum_{p=0}^{n-1}\sum_{\tau\in S_p}\int_{S_\tau}|dW(f_T)|^2,}$$
and for any $\tau\in K$
$$\displaylines{\int_{S_\tau}|dW(f_T)|^2=\int_{S_\tau}|\sum_{\sigma\in St_n(\tau)}d(L_\sigma^{f_T}\phi_\sigma)|^2\hfill\cr
\leq2\int_{S_\tau}\bigl(\sum_{\sigma\in St_n(\tau)}\phi_\sigma|dL^{f_T}_\sigma|\bigr)^2+2\int_{S_\tau}\bigl|\sum_{\sigma\in St_n(\tau)}L^{f_T}_\sigma d\phi_\sigma\bigr|^2.}$$
Since $|dL_\sigma^{f_T}|\leq(1+8(\frac{\Lambda m}{\delta})^2)|L_\sigma|$, and by Proposition \ref{Ssigma}, we have
\begin{align*}
\sum_{p=0}^{n-1}\sum_{\tau\in S_p}2\int_{S_\tau}\bigl(\sum_{\sigma\in St_n(\tau)}\phi_\sigma|dL^{f_T}_\sigma|\bigr)^2\leq\sum_{p=0}^{n-1}\sum_{\tau\in S_p}(1+16(\frac{\Lambda m}{\delta})^2)\sum_{\sigma\in St_n(\tau)}|L_\sigma|^2\Vol S_\tau\\
\leq C(n)\frac{m}{\Diam_M}\sum_{\sigma\in S_n}|L_\sigma|^2\Vol S_\sigma\leq C(n)\frac{m_T}{\Diam_M}q_T(f_T)
.
\end{align*}
To bound the remaining term $\int_{S_\tau}\bigl|\sum_{\sigma\in St_n(\tau)}L^{f_T}_\sigma d\phi_\sigma\bigr|^2$, we set $\sigma_0\in St_n(\tau)$ and $\tau_0$ the realisation of $\tau$ associated to $\exp_{X_{\sigma_0}}$. Then there for any $x\in S_\tau$, there exists $x'\in\tau_0$ such that $d(x,x')\leq\frac{m^2}{\delta}\alpha^{p+1}$. Hence we have $|L^{f_T}_{\sigma_0}(x)-L^{f_T}_{\sigma_0}(x')|\leq2|L_{\sigma_0}|\frac{m^2}{\delta}\alpha^{p+1}$. Note that the barycentric coordinates $(\theta_k)$ of $x'$ in the simplex $\sigma$ satisfy $\theta_k=0$ if $x_{i_{\sigma_0}(k)}\notin\tau$. For any $\sigma\in St_n(\tau)$, we set $x_\sigma$ the point whose barycentric coordinates in $\sigma$ are $\theta'_{i_\sigma^{-1}(x)}=\theta_{i_{\sigma_0}^{-1}(x)}$ if $x$ is a vertex of $\tau$ and $\theta'_k=0$ otherwise. By applying Lemma \ref{coorbar} at most twice, we get $d(x',x_\sigma)\leq200\Lambda^2\frac{m^3}{\delta^2}$. Since $L^{f_T}_\sigma(x_\sigma)=L^{f_T}_{\sigma_0}(x')$, we have that
$$\displaylines{|L^{f_T}_\sigma(x)-L^{f_T}_{\sigma_0}(x)|\leq|L^{f_T}_\sigma(x)-L^{f_T}_\sigma(x_\sigma)|+|L^{f_T}_{\sigma_0}(x')-L^{f_T}_{\sigma_0}(x)|
\leq C(n)(|L_\sigma|+|L_{\sigma_0}|)\frac{m^2}{\delta}.}$$
Since $\sum_\sigma d\phi_\sigma=0$, Lemma \ref{partunit} gives us 
\begin{align*}
\sum_{p=1}^{n-1}\sum_{\tau\in S_p}\int_{S_\tau}&\bigl|\sum_{\sigma\in St_n(\tau)}L^{f_T}_\sigma d\phi_\sigma\bigr|^2=\sum_{p=1}^{n-1}\sum_{\tau\in S_p}\int_{S_\tau}\bigl|\sum_{\sigma\in St_n(\tau)}(L^{f_T}_\sigma-L_{\sigma_0}^{f_T})d\phi_\sigma\bigr|^2\\
&\leq \sum_{p=1}^{n-1}\sum_{\tau\in S_p} C(n)\Theta^{10n^2}\# St_n(\tau)(\sum_{\sigma\in St_n(\tau)}|L_\sigma|^2+|L_{\sigma_0}|^2)\Vol S_\tau\\
&\leq C(n)\Theta^{20n^2}\frac{m_T}{\Diam_M}\sum_{\sigma\in S_n}|L_\sigma|^2\Vol S_\sigma\leq C(n)\Theta^{20n^2}\frac{m_T}{\Diam_M}q_T(f_T) .
\end{align*}

We now prove point (2). As in the proof of point (3), we have that
$$\displaylines{\sum_{\sigma\in S_n}\int_{S_\sigma}(L^{f_T}_\sigma)^2\leq\int_M|W(f_T)|^2\leq\sum_{\sigma\in S_n}\int_{S_\sigma}(L^{f_T}_\sigma)^2+\sum_{p=0}^{n-1}\sum_{\tau\in S_p}\int_{S_\tau}|W(f_T)|^2}$$
and 
\begin{align*}
\sum_{p=0}^{n-1}\sum_{\tau\in S_p}\int_{S_\tau}|W(f_T)|^2&=\sum_{p=0}^{n-1}\sum_{\tau\in S_p}\int_{S_\tau}|\sum_{\sigma\in St_n(\tau)}L_\sigma^{f_T}\phi_\sigma|^2\\
&\leq\sum_{p=0}^{n-1}\sum_{\tau\in S_p}\sum_{\sigma\in St_n(\tau)}C(n)\Theta^{2n}\int_{S_\tau}|L_\sigma^{f_T}|^2=C(n)\Theta^{2n}\sum_{\sigma\in S_n}\int_{\overline{S}_\sigma\setminus S_\sigma}|L_\sigma^{f_T}|^2.
\end{align*}
Hence we have that
$$\Bigl|\int_M|W(f_T)|^2-\sum_\sigma\int_{T_\sigma}(L_\sigma^{f_T})^2\Bigr|\leq C(n)\Theta^{2n}\sum_{\sigma\in S_n}\int_{\overline{S}_\sigma\setminus T_\sigma}|L_\sigma^{f_T}|^2,$$
and since by Theorem \ref{Lipexp}, we have
$$\displaylines{\bigl(1-(\frac{\Lambda m}{\delta})^2\bigr)^n|f_T|_T^2=\bigl(1-(\frac{\Lambda m}{\delta})^2\bigr)^n\sum_\sigma\int_{\Delta_\sigma}(L^\sigma)^2\cr
\leq\sum_\sigma\int_{T_\sigma}(L_\sigma^{f_T})^2\leq\bigl(1+(\frac{\Lambda m}{\delta})^2\bigr)^n\sum_\sigma\int_{\Delta_\sigma}(L^\sigma)^2\leq\bigl(1+(\frac{\Lambda m}{\delta})^2\bigr)^n|f_T|_T^2,}$$
it only remains to bound from above $\int_{\overline{S}_\sigma\setminus S_\sigma}|L_\sigma^{f_T}|^2$. If we set $\zeta=n\Theta^{2n}\alpha^{n+1}\frac{m}{\Diam_M}$, then $\overline{S}_\sigma\setminus T_\sigma\subset(1+\zeta)\Delta_\sigma\setminus(1-\zeta)\Delta_\sigma$ and so
$$\int_{\overline{S}_\sigma\setminus T_\sigma}|L_\sigma^{f_T}|^2\leq\bigl(1+(\frac{\Lambda m}{\delta})^2\bigr)^n\int_{(1+\zeta)\Delta_\sigma\setminus(1-\zeta)\Delta_\sigma}|L^\sigma|^2$$
Let $H_\lambda$ be the dilation in $F$ of factor $\lambda$ and centred at the iso-barycentre of $\Lambda^n$. By the fundamental theorem of the calculus we have
$$\displaylines{\Bigl|\frac{1}{(1-(n+1)\zeta)^n}\int_{(1-\zeta)\Delta_\sigma}(L^\sigma)^2-\int_{\Delta_\sigma}(L^\sigma)^2\Bigr|\leq\Bigl|\int_{\Delta_\sigma}(L^\sigma)^2\bigl(H_{1-(n+1)\zeta}(y)\bigr)-(L^\sigma)^2(y)\Bigr|.}$$
We set  $h=(n+1)\frac{m}{\Diam_M}\zeta$. Since 
$$\dsl{\bigl|(L^\sigma)^2\bigl(H_{1-(n+1)\zeta}(y)\bigr)-(L^\sigma)^2(y)\bigr|\leq2|L^\sigma(y)|\|dL^\sigma\|h\Diam_M+\|dL^\sigma\|^2\Diam_M^2h^2\hfill\cr
\leq C(n)(\frac{m}{\Diam_M})^2\bigl(|L^\sigma(y)|^2+\Diam_M^2\|dL^\sigma\|^2\bigr)}$$
then we have
$$\dsl{\Bigl|\frac{1}{(1-(n+1)\zeta)^n}\int_{(1-\zeta)\Delta_\sigma}(L^\sigma)^2-\int_{\Delta_\sigma}(L^\sigma)^2\Bigr|\leq C(n)(\frac{m}{\Diam_M})^2\int_{\Delta_\sigma}(L^\sigma)^2+\Diam_M^2\|L^\sigma\|^2,}$$
By the same way we get
$$\Bigl|\frac{1}{(1+(n+1)\zeta)^n}\int_{(1+\zeta)\Delta_\sigma}(L^\sigma)^2-\int_{\Delta_\sigma}(L^\sigma)^2\Bigr|\leq C(n)(\frac{m}{\Diam_M})^2\int_{\Delta_\sigma}(L^\sigma)^2+\Diam_M^2\|L^\sigma\|^2$$
Hence $\sum_{\sigma S_n}\int_{(1+\zeta)\Delta_\sigma\setminus(1-\zeta)\Delta_\sigma}|L^\sigma|^2\leq C(n)(\frac{m}{\Diam_M})^2 (|f_T|_T^2+\Diam_M^2q_T(f_T))$,
which combined with the previous inequalities gives the result.
\end{proof}

\subsection{Conclusion}
\label{compspec}

By Propositions \ref{DiscL2}, \ref{Diri1} and \ref{min-max} we can almost bound from below the first eigenvalues of $(M^n,g)$ by the eigenvalues of $q_T$ with respect to $\<\cdot,\cdot\>_T$. To have an error bound that depends on $p$ and not on $\lambda_p$, we use the following rough version of a well known result due to S. Cheng.

\begin{lemma}\label{mLi}
 Let $(M^n,g)$ be a compact Riemannian manifold such that $\Diam_M^2|\sigma|\leq\Lambda^2$. For any $k\in\N$, we have $\Diam_M^2\lambda_k\leq C(n)(\frac{\Diam_M}{\inj})^2 e^\frac{ne^\Lambda}{2}k^2$.
\end{lemma}

We infer the following theorem.

\begin{theorem}\label{MinLi}
 Let $\epsilon\in]0,1[$ be a real number, $(M^n,g)$ be a compact Riemannian manifold such that $\Diam_M^2|\sigma|\leq\Lambda^2$ and $T$ be a geodesic triangulation of $(M^n,g)$ such that $\frac{m_T}{\Diam_M}\leq\epsilon\Bigl(\frac{C(n)\inj}{\Diam_M\Theta_Te^{e^\Lambda}p}\Bigr)^{3n^3}$, then we have
$\lambda_p(q_T)\leq\lambda_p(M^n,g)(1+\epsilon)$.
\end{theorem}

Once again, by Propositions \ref{EstimW} and \ref{min-max} we can bound from above the eigenvalues of $(M^n,g)$ by the eigenvalues of $q_T$ with respect to $\<\cdot,\cdot\>_T$. Note that to bound $\lambda_p(q_T)$, we just have to use Theorem \ref{MinLi}.

\begin{theorem}\label{MaxLi}
   Let $\epsilon\in]0,1[$ be a real number, $(M^n,g)$ be a compact Riemannian manifold such that $\Diam_M^2|\sigma|\leq\Lambda^2$ and $T$ be a geodesic triangulation of $(M^n,g)$ such that $\frac{m_T}{\Diam_M}\leq\epsilon\Bigl(\frac{C(n)\inj}{\Diam_M\Theta_Te^{e^\Lambda}p}\Bigr)^{3n^3}$. Then, we have that
$$\lambda_p(M^n,g)\leq\lambda_p(q_T)(1+\epsilon).$$
\end{theorem}

\begin{proof}
  When applying Proposition \ref{min-max} with $E_1$ the space spanned by the $p+1$ first eigenfunction of $q_T$, Proposition \ref{EstimW} gives $\alpha\geq 1-C(n)e^\Lambda\Theta^{2n}(\frac{m_T}{\Diam_M})^2\bigl(1+\Diam_M^2\lambda_p(q_T)\bigr)$ and $\beta\leq1+C(n)\Theta^{20n^2}\frac{m}{\Diam_M}$. Now, by Theorem \ref{MinLi} and Lemma \ref{mLi}, we have $\Diam_M^2\lambda_p(q_T)\leq\frac{C(n)\Diam_M^2e^\frac{ne^\Lambda}{2}p^2}{\inj^2}$.
\end{proof}

\section{Proof of Theorem \ref{approvecp}}
\label{compeigenf}

\subsection{Approximation of the eigenfunctions}\label{prelalg}
\bigskip

To get the relations between the eigenfunctions of $(M^n,g)$ and the discrete eigenfunctions, we first prove the following result, where the notations are the same as in the introduction.

\begin{lemma}\label{eigenf}
Let $\delta>0$ and assume that $\lambda_p(M)+\eta\leqslant\lambda_{p+1}(M)$. 

For any $f\in E_p$, we have
$\|R(f)-P_p\circ R(f)\|_T^2\leqslant \frac{C(p,n,\Lambda,\frac{\Diam_M}{i_0})}{\sqrt{\eta}}(\frac{m_T}{\Diam_M})^\frac{1}{6n^2}\|R(f)\|_T^2$
where $P_p$ is the orthogonal projection from $\R^N$ to the space $F_p$ spanned by the first $p$ eigenfunctions of $q_T$.

For any $(y_i)\in F_p$, we have
$\|W(y_i)-Q_p\circ W(y_i)\|_T^2\leqslant C(p,n,\Lambda,\frac{\Diam_M}{i_0},\eta)(\frac{m_T}{\Diam_M})^\frac{1}{6n^2}\|W(y_i)\|_T^2$,
where $Q_p$ is the orthogonal projection from $L^2(M)$ to $E_p$.
\end{lemma}

\begin{proof}
We use the same idea as in \cite{CV}. We consider in $\Lambda^{p+1}\R^N$ the operator $A(v_0\wedge\cdots\wedge v_p)=\sum_{i=0}^pv_0\wedge\cdots\wedge\Delta_T(v_i)\wedge\cdots\wedge v_p$, where $\Delta_T$ is the symmetric operator such that $q_T(x)=\langle \Delta_T(x),x\rangle_T$. The eigenvalues of $A$ are exactly the sum $\lambda_{i_1}(T)+\cdots+\lambda_{i_{p+1}}(T)$ with $0\leqslant i_1<\cdots<i_{p+1}\leqslant N$, and so its first eigenvalue is $\lambda_{0}(T)+\cdots+\lambda_{p}(T)$. We set $\overline{R}:\Lambda^{p+1} E_p\to\Lambda^{p+1} \R^N$ defined by $\overline{R}(f_0\wedge\cdots\wedge f_p)=R(f_0)\wedge\cdots\wedge R(f_p)$. Since $|\overline{R}(f_0\wedge\cdots\wedge f_p)|_T^2=(\det G)^2$, where $G$ is the Gramm matrix of the family $R(f_0),\cdots, R(f_p)$, and since there exists $C_p$ such that $|(\det G)^2-1|\leqslant C_p|G-I_{p+1}|$ for $G$ near $I_{p+1}$, Proposition \ref{DiscL2} and Lemma \ref{mLi} give us
\begin{align*}
\bigl||\overline{R}(f_0\wedge\cdots\wedge f_p)|_T^2-1\bigr|\leqslant C(p,n,\Theta,\Lambda,\frac{\Diam_M}{\inj})(\frac{m_T}{\Diam_M})^2
\end{align*}

By Propositions \ref{DiscL2} and \ref{Diri1} and by Theorem \ref{approvalp}, we get
\begin{align*}
&\bigl|\bigl\langle A\overline{R}(f_0\wedge\cdots\wedge f_p),\overline{R}(f_0\wedge\cdots\wedge f_p)\bigr\rangle-\sum_{i=0}^p\lambda_i(T)|\overline{R}(f_0\wedge\cdots\wedge f_p)\bigr|^2\bigr|\\
&=\bigl|\sum_{i=0}^p\bigl\langle R(f_0)\wedge\cdots\wedge(A(R(f_i))-\lambda_i(T)R(f_i))\wedge\cdots\wedge R(f_p),R(f_0)\wedge\cdots\wedge R(f_p)\bigr\rangle\bigr|\\
&=\bigl|\sum_{i,j=0}^p\bigl\langle R(f_0)\wedge\cdots\wedge\langle A(R(f_i))-\lambda_i(T)R(f_i)),R(f_j)\rangle G^{ij}R(f_j)\wedge\cdots\wedge R(f_p),R(f_0)\wedge\cdots\wedge R(f_p)\bigr\rangle\bigr|\\
&=\bigl|\sum_{i=0}^pG^{ii}(\det G)^2(q_T(R(f_i))-\lambda_i(T)|R(f_i)|_T^2)\bigr|\\
&\leqslant\bigl|\sum_{i=0}^pG^{ii}(\det G)^2(q_T(R(f_i))-\|df_i\|^2)\bigr|+\bigl|\sum_{i=0}^pG^{ii}(\det G)^2(\lambda_i(M)-\lambda_i(T))\bigr|\\
&+\bigl|\sum_{i=0}^pG^{ii}(\det G)^2\lambda_i(T)(1-|R(f_i)|_T^2)\bigr|\leqslant C(p,n,\Theta,\Lambda,\frac{\Diam_M}{\inj})(\frac{m_T}{\Diam_M})^\frac{1}{3n^2}
\end{align*}
Let $(y_i)$ an orthonormal family of eigenfunctions of $q_T$ associated to the eigenvalue $(\lambda_i(T))$. We set $\overline{R}(f_0\wedge\cdots\wedge f_p)=\alpha y_0\wedge\cdots\wedge y_p+n$ with $n$ orthogonal to $y_0\wedge\cdots\wedge y_p$. The above estimates give us
$|\alpha^2+|n|^2-1|\leqslant C(p,n,\Theta,\Lambda,\frac{\Diam_M}{\inj})(\frac{m_T}{\Diam_M})^2$ and $\delta|n|^2\leqslant C(p,n,\Theta,\Lambda,\frac{\Diam_M}{\inj})(\frac{m_T}{\Diam_M})^\frac{1}{3n^2}$, from which we easily get
$$\sum_{i=0}^p\|R(f_i)-P_p(R(f_i))\|_T^2\leqslant|\overline{R}(f_0\wedge\cdots\wedge f_p)-y_0\wedge\cdots\wedge y_p|^2\leqslant \frac{C(p,n,\Theta,\Lambda,\frac{\Diam_M}{\inj})(\frac{m_T}{\Diam_M})^\frac{1}{6n^2}}{\sqrt{\eta}}$$

The proof of the other estimate is exactly the same, but we first have to bound from below the gap $\lambda_{p+1}(T)-\lambda_p(T)$ using the bound on the gap $\lambda_{p+1}(M)-\lambda_p(M)$ and Theorem \ref{approvalp}. 
\end{proof}

We easily infer Theorem \ref{approvecp} from the previous Lemma. Indeed, by Proposition \ref{DiscL2}, the quadratic form $|P_p\circ R|^2$ on $E_q$ takes only values less than $1+C(\frac{m_T}{\Diam_M})^\frac{1}{6n^2}$ on the unit sphere of $E_q$ and so its trace with respect to $\langle\cdot,\cdot\rangle_T$ is less than $p+C(\frac{m_T}{\Diam_M})^\frac{1}{6n^2}$ (complete an orthonormal basis of the kernel of $P_p\circ R$). But the previous lemma, applied for the spectral gap at $p$, implies that the trace of $|P_p\circ R|^2$ on $E_p$ is close to $p$ and so $p-C(\frac{m_T}{\Diam_M})^\frac{1}{6n^2}+\sum_{p+1\leqslant i\leqslant q}|P_p\circ R(f_i)|^2={\rm tr} |P_p\circ R|^2\leqslant p+C(\frac{m_T}{\Diam_M})^\frac{1}{6n^2}$, and so $\sum_{p+1\leqslant i\leqslant q}|P_p\circ R(f_i)|^2\leqslant C(\frac{m_T}{\Diam_M})^\frac{1}{6n^2}$. This gives the result when combined with the previous lemma applied to the spectral gap at $q$.

\end{document}